\documentclass[reqno, 10pt, a4paper]{amsart}

\usepackage{
    a4wide,
    amssymb,
    bbm,
    centernot,
    mathtools,
} 
\usepackage{float}
\usepackage{graphicx}
\usepackage[cal=euler, scr=boondoxo]{mathalfa}

\usepackage{amsmath,amsthm,amssymb}

\usepackage{caption}
\usepackage{subcaption}

\usepackage{tikz}
\usepackage{enumitem}

\usepackage{ stmaryrd }

\usepackage{hyperref}
\hypersetup{allcolors=blue,
pdftitle={Amnesia-reinforced random walk},
    pdfpagemode=FullScreen,colorlinks=true, linkcolor=blue,citecolor=blue} %

\numberwithin{equation}{section}
\usepackage{bbm}

\usepackage{braket}
\usepackage{physics}
\usepackage{bm}
\usepackage{esvect}
\usepackage[english]{babel}
\usepackage{extarrows}

\usepackage{setspace}
\newtheorem{theorem}{Theorem}[section]
\newtheorem{lemma}{Lemma}[section]
\newtheorem{prop}{Proposition}[section]

\theoremstyle{definition}

\theoremstyle{remark}
\newtheorem{remark}{Remark}[section]

\numberwithin{equation}{section}

\newcommand{\bbP}{\ensuremath{\mathbb P}}

\begin{document}

\title[Multidimensional amnesia-reinforced elephant random walk]{Analysis of the smoothly amnesia-reinforced  multidimensional elephant random walk}


\author{Jiaming Chen}
\address{Departement Mathematik, ETH Zürich}
\curraddr{101, Rämistrasse, CH-8092 Zürich, Switzerland}
\email{jiamchen@student.ethz.ch}

\author{Lucile Laulin}
\address{Laboratoire de Mathématiques Jean Leray, Nantes Université}
\curraddr{2 Chem. de la Houssinière, 44322 Nantes, France}
\email{lucile.laulin@math.cnrs.fr}
\thanks{}

\begin{abstract}
In this work, we discuss the smoothly amnesia-reinforced multidimensional elephant random walk (MARW). The scaling limit of the MARW is shown to exist in the diffusive, critical and superdiffusive regimes. We also establish the almost sure convergence in all of the three regimes. The quadratic strong law is displayed in the diffusive regime as well as in the critical regime. The mean square convergence towards a non-Gaussian random variable is established in the superdiffusive regime. Similar results for the barycenter process are also derived. Finally, the last two Sections are devoted to a discussion of the convergence velocity of the mean square displacement and the Cramér moderate deviations.
\end{abstract}

\dedicatory{}
\maketitle

\tableofcontents

\section{Introduction}
\label{sec:intro}
{
The study of reinforced processes and reinforced random walks has known a growing interest over the last decades. In particular, random walks on graphs, or more precisely edge \cite{Merkl/Rolles} or vertex \cite{Pemantle1992} reinforced random walks, have been the subject of a great number of contributions,  see also \cite{baur2019, Bertoin2020, Kozam2013} and the references therein. The insight of introducing reinforcement mechanisms to stochastic processes has also shed light on more applied models. In \cite{Lamberton/Pages/Tarres2004}, the adaptive strategy of an agent who plays a two-armed bandit machine was described as a self-reinforced random walk. The philosophy of stochastic reinforcement has also been discussed in the topics of evolutionary ecology \cite{Benaim/Schreiber/Tarres} and machine learning theory \cite{DaiPra/Louis/Minelli}. Another manifestation of reinforced Pólya urn models on financial economics can be found in \cite{Marcaccioli/Livan}. We also refer the readers to \cite{Pemantle2007} for a comprehensive and extensive survey on the subject. 

The Elephant Random Walk (ERW) is a discrete-time random walk, introduced by Schütz and Trimper \cite{Schutz2004} in 2004. 
It was referred to as the ERW in allusion to the traditional saying that elephants can always remember anywhere they have been.  As it was pointed out \cite{Bertoin2020} by Bertoin who relied on Kürsten's work \cite{Kursten2016}, the ERW is a special case of {\itshape step-reinforced random walk}. In fact, the ERW is reinforced because its behavior is influenced by its past : the ERW may have a tendency to do the same thing over and over, or on the contrary, it may try to compensate its previous steps. This different types of behavior, here-called regimes, are ruled by the memory parameter $p$ and it is well-known that the ERW shows three regimes of behavior and that the critical value is $p=3/4$.

The ERW in dimension $d=1$ has received a lot of attention from mathematicians and physicists over the last two decades. The almost sure convergence and the asymptotic normality of the position of the ERW  were established in the diffusive regime $p< 3/4$ and the critical regime $p=3/4$, see \cite{Baur,Bercu,Coletti} and the references therein. In the superdiffusive regime $p>3/4$, Bercu \cite{Bercu2} proved that the limit of the position of the ERW is not Gaussian and Kubota and Takei \cite{Kubota2019} showed that the fluctuation of the ERW around this limit is Gaussian. To obtain those asymptotics, various approaches have been followed : Baur and Bertoin \cite{Baur} went with the connection to P\'{o}lya-type urns while martingales were used by Bercu \cite{Bercu2} and Coletti et al. \cite{Coletti} and the construction of random trees with Bernoulli percolation have been explicited by Kürsten \cite{Kursten2016} and Businger \cite{Businger2018}. 

Other quantities of interest regarding the ERW have been studied. For example, Fan et al. \cite{Fan} provided the Cramer moderate deviations associated with the ERW in dimension 1 and, more recently, Hayashi et al. \cite{Hayashi} studied the rate of quadratic mean displacement.

Bercu and Laulin \cite{Bercu} introduced the multidimensional ERW (MERW), where $d \geq 1$, and established the natural extensions of the results \cite{Bercu2} in dimension $d=1$. Then, they investigated the center of mass of the MERW \cite{Bercu3}. In both papers, they extensively used a martingale approach. Bertenghi \cite{Bertenghi} made use of the connection to Pólya-type urns in order to establish functional results for the MERW.  

Finally, the ERW with changing memory has also been introduced. The ERW with linearly reinforced memory has been studied by Baur \cite{baur2019} via the urn approach, and Laulin \cite{Laulin1} using martingales. Gut and Stadmüller \cite{GutA2021} proposed an amnesic ERW where the elephant could stop and only remember the first (and second) step it tooks. They also investigated  the case where the elephant only remembered a fixed or time-evolving portion of its past (recent or distant) \cite{GutStad2021}. In the recent work \cite{Laulin2}, Laulin introduced smooth amnesia to the memory of the ERW and established the asymptotic behavior of this new process.
\medskip

The idea of our paper is to generalise the work \cite{Laulin2} in dimension 1 to the dimension $d\geq1$. In other words, we introduce smooth amnesia to the memory of the multidimensional elephant random walk.

\begin{figure}
\begin{subfigure}{0.33\textwidth}
    \includegraphics[width=\textwidth]{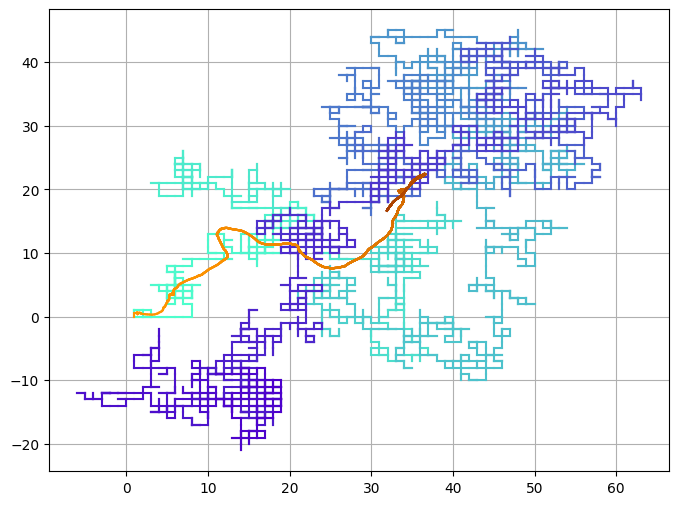}
    \caption{Diffusive regime}
    \label{AERW-picture-diff}
\end{subfigure}\hfill
\begin{subfigure}{0.33\textwidth}
    \includegraphics[width=\textwidth]{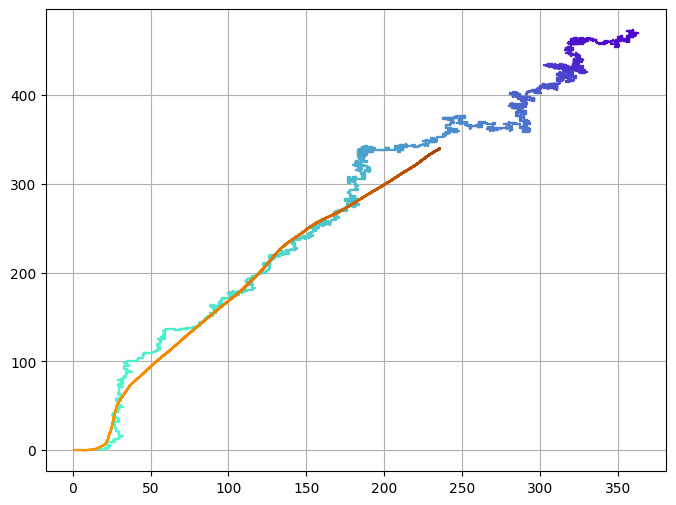}
    \captionof{figure}{Critical regime}
    \label{AERW-picture-crit}
\end{subfigure}
\hfill
\begin{subfigure}{0.33\textwidth}
    \includegraphics[width=\textwidth]{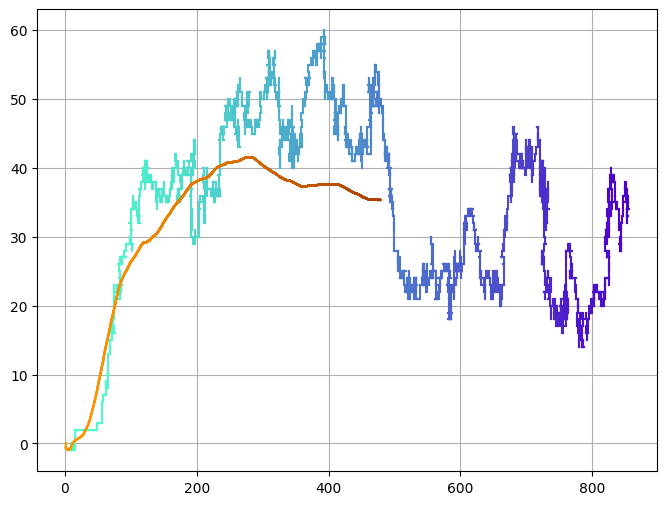}
    \captionof{figure}{Superdiffusive regime}
    \label{AERW-picture-sdiff}
\end{subfigure}
\caption{The two-dimensional ERW with amnesia (in blue) and its barycenter (in red).}
\label{fig-MAERW}
\end{figure}
\smallskip

Our paper is organized as follows.}
In Section \ref{sec:chapter 1}, we introduce the basic setting of the elephant random walk $(S_n)_{n\in\mathbb{N}}$ placed under an amnesia reinforcement mechanism, which is controlled by the memory sequence $(\beta_n)_{n\in\mathbb{N}}$. This type of multidimensional reinforced random walked is named as the multidimensional amnesia-reinforced elephant random walk (MARW). { Similar to the ERW with the amnesia reinforcement, the MARW also admits a martingale structure, which is discussed in Section \ref{sec:chapter 2}. Unlike the usual ERW, the additional amnesia-reinforcement induces two discrete-time martingales, instead of a single martingale, which are strongly correlated in a nontrivial fashion.} Such strong correlation of martingales will eventually pose some computational difficulties when we analyze the limiting behavior of the MARW in Section \ref{sec:chapter 3}. For instance, when we compute the pointwise limit and the scaling limit of $(S_n)_{n\in\mathbb{N}}$ in the diffusive regime, the two strongly correlated martingales have to be dealt with separately, see \cite{Bercu3,Laulin1,Laulin2} for the same methodology. \par
As a courtesy to our readers, we give a preview of some of our main results, whose proofs will be deferred to Theorem \ref{thm: 4.1.1}, Theorem \ref{thm: 4.1.2}, and Theorem \ref{thm: 4.1.3}. In the diffusive regime, we have the almost sure convergence,
\begin{equation*}
    \frac{1}{n}S_n\to0\quad \text{as}\quad n\to\infty\quad \mathbb{P}-a.s.
\end{equation*}
Another logarithmic scaling to the MARW yields the quadratic strong law,
\begin{equation*}
        \frac{1}{\log n}\sum\limits_{k=1}^n\frac{S_kS_k^T}{k^2}\to C(p,(\beta_n)_{n\in\mathbb{N}})\cdot\frac{1}{d}I_d\quad \text{as}\quad n\to\infty\quad \mathbb{P}\text{-a.s.}
\end{equation*}
where the constant $C(p,(\beta_n)_{n\in\mathbb{N}})>0$ depends only on the parameter $p$ and the control sequence $(\beta_n)_{n\in\mathbb{N}}$ of the amnesia-reinforcement. Using square-root scaling factor, we observe that the MARW also admits a scaling limit in the diffusive regime, or convergence in distribution, in the Skorokhod space $\mathfrak{D}(\mathbb{R}_+)$ of càdlàg functions, in the sense that
\begin{equation*}
        \bigg(\frac{1}{\sqrt{n}}S_{\lfloor nt\rfloor},\;t\geq0\bigg)\Longrightarrow\bigg(W_t,\;t\geq0\bigg)
\end{equation*}
where $(W_t)_{t\geq0}$ is a continuous $\mathbb{R}^d$-valued centered Gaussian process such that $W_0=0$ and with covariance structure given in \eqref{4.42}.\par
It is also of interest to look at the barycenter process $(G_n)_{n\in\mathbb{N}}$ of the MARW. Its definition as well as its limiting behavior are discussed in Section \ref{sec: barycenter}. Similar to the discussion of the MARW, we obtain its pointwise convergence, quadratic strong law, and its scaling limit. In particular, Theorem \ref{thm: 5.3.1} states that the barycenter process admits a scaling limit at the diffusive regime, or convergence in distribution, in the Skorokhod space $\mathfrak{D}([0,1])$ of càdlàg functions, such that
\begin{equation*}
        \bigg(\frac{1}{\sqrt{n}}G_{\lfloor nt\rfloor},\;t\geq0\bigg)\Longrightarrow\bigg(\int\limits_{0}^1 W_{tr}\,dr,\;t\geq0\bigg)
\end{equation*}
where $(W_t)_{t\geq0}$ is a continuous $\mathbb{R}^d$-valued centered Gaussian process defined in Theorem $\ref{thm: 4.1.3}$ with its covariance structure defined in \eqref{4.42}.\par
A natural question to ask is how fast the limiting Theorems in Section \ref{sec:chapter 3} are carried on. Section \ref{sec: rate} provides a quantitative estimate on the mean square convergence velocity of the pointwise limit, quadratic strong law, and the scaling limit of the MARW. It should be possible to derive similar convergence velocity to the barycenter process, which is not computed in this work. In Section \ref{sec: Cramér moderate deviations}, we end this work with a discussion on the Cramér moderate deviations of the MARW in the diffusive and critical regimes. As a preview of our result in this Section, let $(\vartheta_n)_{n\in\mathbb{N}}\subseteq\mathbb{R}$ be a non-decreasing sequence so that $\vartheta_n/\sqrt{n}\to0$ as $n\to\infty$, and $w_n$ the sequence with asymptotic behavior described in Lemma \ref{LEM:wn-regimes}. Take any non-empty Borel set $B\subseteq\mathbb{R}^d$, then we have 
\begin{equation}\begin{aligned}
        -\inf\limits_{x\in\text{int}\,B}\frac{1}{2}\norm{x}^2&\leq\liminf\limits_{n\to\infty}\vartheta_n^{-2}\log\mathbb{P}\bigg(\frac{a_n\mu_n S_n}{\vartheta_n\sqrt{w_n}}\in B\bigg)\\
        &\leq\limsup\limits_{n\to\infty}\vartheta_n^{-2}\log\mathbb{P}\bigg(\frac{a_n\mu_n S_n}{\vartheta_n\sqrt{w_n}}\in B\bigg)\leq-\inf\limits_{x\in\text{cl}\,B}\frac{1}{2}\norm{x}^2,
\end{aligned}\end{equation}
where $\text{int}\,B$ and $\text{cl}\,B$ denote the interior and the closure of $B\subseteq\mathbb{R}^d$, respectively. This is the Cramér moderate deviations for the MARW in the diffusive and critical regimes.\par
Moreover, we chose to postpone some technicalities regarding the analysis of the random walk to the Appendix \ref{appendix}. That way, the reader can focus on the main Theorems and the ideas of their proofs. However, some analogous technicalities are displayed in the proof of the Theorems on the barycenter such that the reader can also have a complete overview of the work needed. \par
Other probabilistic aspects of interest to the MARW include the statistical inference and an analysis on the Fisher information, see \cite{Bercu/Laulin}, as well as the Wasserstein distance of the reinforced random walk, see \cite{Fan3}. Perturbations of the amnesia intensity and its stability for the MARW is also of independent interest. A similar topic for another type of stochastic process, the Schramn-Loewner evolution, has been considered in \cite{Bauer/Bernard/Kytola, Chen/Margarint}. The transience and recurrence property of the MARW remains unknown, to the best of our knowledge. Readers are referred to \cite{Bertoin2021, Fan} for an exposition on the ERW without the amnesia reinforcement mechanism.


\section{The amnesia-reinforced elephant random walk}
\label{sec:chapter 1}
To begin with, let us properly introduce the MARW. It is the natural extension to higher dimensions of the one-dimensional MARW, defined in \cite{Laulin1}. For an arbitrarily given dimension $d\geq1$, let $(S_n)_{n\in\mathbb{N}}$ be a (reinforced) random walk on $\mathbb{Z}^d$ starting from the origin at time $n=0$, i.e.~$S_0=0$. At time $n=1$, the reinforced random walk moves to one of the $2d$ nearest-neighbors with equal probability $1/2d$. After that, at time $n\geq1$, the reinforced random walk chooses at random an integer $1\leq k\leq n$ among the past times and performs the same step with probabily $p$, or goes in any of the $2d-1$ other directions with probability $(1-p)/(2d-1)$. This random walk possesses the amnesia property, in the sense that it remembers its most recent past steps better than its remote past steps. Colloquially, this random walk has higher probability to choose its recent steps than its earlier steps.\par
From a mathematical perspective, the position of this reinforced random walk at time $n+1\geq1$ is given by
\begin{equation*}
    S_{n+1}=S_n+X_{n+1}
\end{equation*}
with $X_{n+1}$ being defined as the step of this random walk at time $n+1$, satisfying {
\begin{equation*}
    X_{n+1}=A_{n+1}X_{\beta_{n+1}}.
\end{equation*}
Here $A_{n+1}$ is a random $d\times d$ matrix given by
\begin{equation*}
    \bbP(A_n = + I_d)=p, 
\end{equation*}
and, for all $1\leq k\leq d-1$,
\begin{equation*}
  \bbP(A_n = - I_d) = \bbP(A_n = + J_d^k)= \bbP(A_n = - J_d^k) = \frac{1-p}{2d-1}
\end{equation*}
where $I_d$ is the identity matrix of order $d$, $I_d = (\delta_{i,j})_d$ and $J_d = C(0,1,0,\ldots,0)$ is the circulant matrix of order $d$ such that $J=(\delta_{i+1,j})_d$.
It is easy to observe that the fixed permutation matrix $J_d$ satisfied $J_d^d=I_d$. The distribution of the memory $\beta_n$ of the reinforced random walk is such that the probability of choosing a fixed past time $k\in\mathbb{N}$ decays approximately with rate $k^\beta/n^{\beta+1}$, where $\beta\geq0$ is the amnesia parameter. 
\begin{figure}[H]
\begin{subfigure}{0.47\textwidth}
    \includegraphics[width=\textwidth]{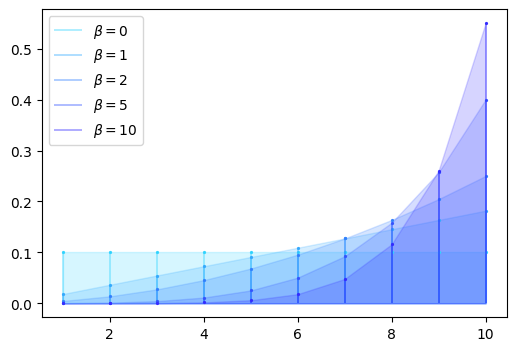}
    \caption{$n=10$}
    \label{AERW-picture-memory10}
\end{subfigure}\hfill
\begin{subfigure}{0.47\textwidth}
    \includegraphics[width=\textwidth]{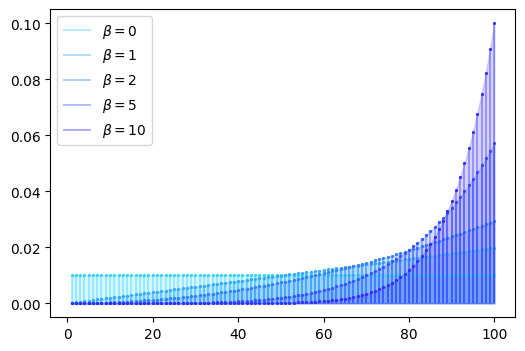}
    \captionof{figure}{$n=100$}
    \label{AERW-picture-memory100}
\end{subfigure}
\caption{Evolution of the distribution of the memory $\beta$ depending on the value of $\beta$ and the time.}
\label{fig-memory}
\end{figure}
To be precise, this random walk chooses $\beta_{n+1}$ according to}
\begin{equation*}
    \mathbb{P}\big(\beta_{n+1}=k\big)=\frac{(\beta+1)\Gamma(\beta+k)\Gamma(n)}{\Gamma(k)\Gamma(\beta+n+1)}=\frac{\beta+1}{n}\cdot\frac{\mu_k}{\mu_{n+1}}\quad \text{for all}\quad 1\leq k\leq n,
\end{equation*}
where 
\begin{equation}
    \label{2.6}
    \mu_n=\prod\limits_{k=1}^{n-1}\bigg(1+\frac{\beta}{k}\bigg)=\frac{\Gamma(\beta+n)}{\Gamma(n)\Gamma(\beta+1)}.
\end{equation}

\begin{figure}[H]
\begin{subfigure}{0.23\textwidth}
    \includegraphics[width=\textwidth]{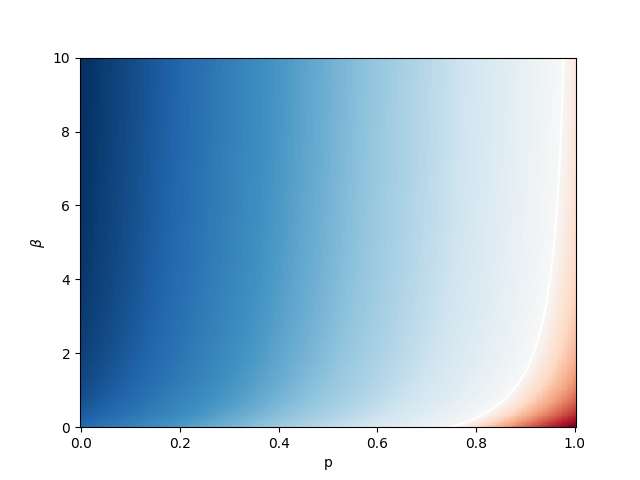}
    \caption{$d=1$}
    \label{AERW-picture-d1}
\end{subfigure}\hfill
\begin{subfigure}{0.23\textwidth}
    \includegraphics[width=\textwidth]{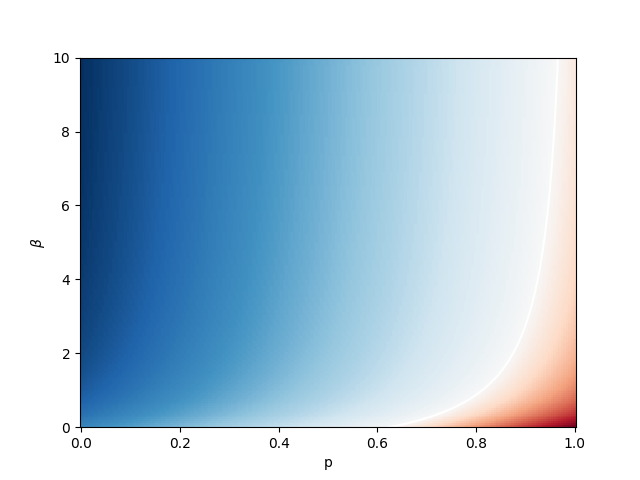}
    \captionof{figure}{$d=2$}
    \label{AERW-picture-d2}
\end{subfigure}
\hfill
\begin{subfigure}{0.23\textwidth}
    \includegraphics[width=\textwidth]{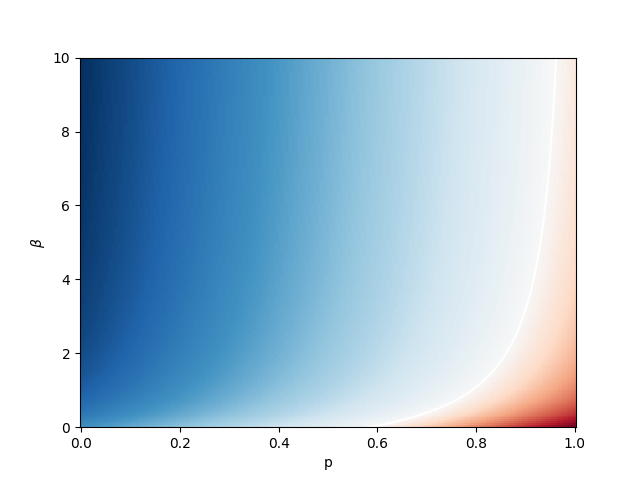}
    \captionof{figure}{$d=3$}
    \label{AERW-picture-d3}
\end{subfigure}
\begin{subfigure}{0.23\textwidth}
    \includegraphics[width=\textwidth]{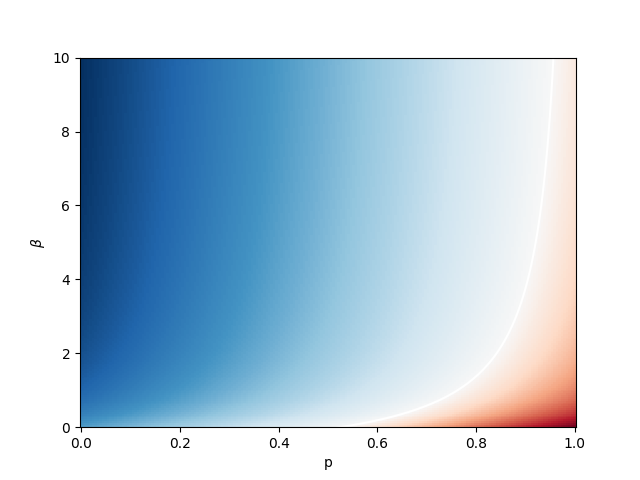}
    \captionof{figure}{$d=10$}
    \label{AERW-picture-d10}
\end{subfigure}
\caption{Competition between the dimension and the amnesia. }
\label{fig-competition}
\end{figure}
{ Figure \ref{fig-competition} aims to give a better understanding on how amnesia affects the MARW in various dimensions. The horizontal axis corresponds to $p$ (from 0 to 1) and the vertical axis corresponds to $\beta$ (from 0 to 10, arbitrary chosen). The diffusive regime, ie. when $p<\frac{4d\beta+2d+1}{4d(\beta+1)}$ or $a<1-\frac{1}{2(\beta+1)}$, is in blue while the superdiffusive regime is in red, see Lemma \ref{LEM:wn-regimes} for the definition of the regimes. One can observe that when the amnesia parameter $\beta$ grows, the superdiffusive regime tends to be less represented. It should also be noted that when the dimension grows the superdiffusive regime is more important. Hence, the amnesia is somehow leading the MARW to a behavior closer to the one in dimension 1. When $\beta$ vanishes, i.e.~$\beta=0$, the MARW reduces to the multidimensional elephant random walk (MERW) introduced in \cite{Bercu}.
\smallskip
\par
The two random variables $A_n$ and $\beta_n$ are constructed to be conditionally independent.}   At each time $n$, define the $\sigma$-algebra $\mathcal{F}_n=\sigma(X_1,\ldots,X_n)$. Then $(\mathcal{F}_n)_{n\in\mathbb{N}}$ is a discrete-time filtration to which the MARW is clearly adapted.\par
Since $A_n$ and $\beta_n$ are conditionally independent, we clearly have
\begin{equation}\begin{aligned}
\label{2.7}
    &\mathbb{E}\big[X_{n+1}|\mathcal{F}_n\big]=\mathbb{E}\big[A_n\big]\mathbb{E}\big[X_{\beta_{n+1}}|\mathcal{F}_n\big]\\
    &\quad \quad =\frac{2dp-1}{2d-1}\mathbb{E}\big[\sum\limits_{k=1}^nX_{k}\mathbbm{1}_{\{\beta_{n+1}=k\}}|\mathcal{F}_n\big]=\frac{2dp-1}{2d-1}\cdot\frac{\beta+1}{n\mu_{n+1}}\sum\limits_{k=1}^n\mu_kX_k.
\end{aligned}\end{equation}
We further denote
\begin{equation}
\label{2.8}
    a=\frac{2dp-1}{2d-1}\quad \;\text{and}\quad \;Y_n=\sum\limits_{k=1}^n\mu_kX_k
\end{equation}
such that
\begin{equation*}
    \mathbb{E}\big[Y_{n+1}|\mathcal{F}_n\big]=\bigg(1+\frac{a(\beta+1)}{n}\bigg)Y_n=\gamma_nY_n
\end{equation*}
with $\gamma_n=1+a(\beta+1)/n$. Hereafter, for each $n\geq1$, let
\begin{equation}
\label{2.10}
    a_n=\prod_{k=1}^{n-1}\gamma^{-1}_k=\frac{\Gamma(n)\Gamma(a(\beta+1)+1)}{\Gamma(a(\beta+1)+n)}\quad \;\text{and}\quad \;w_n=\sum\limits_{k=1}^n(a_k\mu_k)^2.
\end{equation}
From a Gamma function estimate, also see in \cite{Laulin1}, we have that
\begin{equation}
\label{2.11}
    n^{a(\beta+1)} a_n\to\Gamma(a(\beta+1)+1)\quad \text{as}\quad n\to\infty
\end{equation}
and
\begin{equation}
\label{2.12}
    n^{-\beta}\mu_n\to{\Gamma(\beta+1)^{-1}}\quad \text{as}\quad n\to\infty.
\end{equation}

{

\section{A correlated martingale approach}}\label{sec:chapter 2}
Define the following two $\mathbb{R}^d$-valued processes by
\begin{equation}
\label{2.13}
    M_n=a_nY_n\quad \;\text{and}\quad \;N_n=S_n+\frac{a(\beta+1)}{\beta-a(\beta+1)}\mu_n^{-1}Y_n.
\end{equation}
\begin{prop}
    The $\mathbb{R}^d$-valued processes $(M_n)_{n\in\mathbb{N}}$ and $(N_n)_{n\in\mathbb{N}}$ defined in \eqref{2.13} are locally square-integrable martingales adapted to $(\mathcal{F}_n)_{n\in\mathbb{N}}$.
\end{prop}
\begin{proof}
    Since, both $M_n$ and $N_n$ are finite sums for each $n\geq1$, the square-integrability and adaptness are immediate. By \eqref{2.8} and \eqref{2.10}, we have
    \begin{equation*}
        \mathbb{E}\big[M_{n+1}|\mathcal{F}_n\big]=a_n\gamma_n^{-1}Y_n+a_n\mu_n\gamma^{-1}_n\mathbb{E}\big[X_{n+1}|\mathcal{F}_n\big]=a_nY_n.
    \end{equation*}
    And by \eqref{2.7}, we have
    \begin{equation*}
        \mathbb{E}\big[N_{n+1}|\mathcal{F}_n\big]=\mathbb{E}\bigg[S_{n+1}+\frac{a(\beta+1)}{\beta-a(\beta+1)}\mu_{n+1}^{-1}Y_{n+1}|\mathcal{F}_n\bigg]=S_n+\frac{a(\beta+1)}{\beta-a(\beta+1)}\mu_n^{-1}Y_n.
    \end{equation*}
Hence the assertion is verified.
\end{proof}
Notice that via introducing the martingales $(M_n)_{n\in\mathbb{N}}$ and $(N_n)_{n\in\mathbb{N}}$, we can write $S_n$ as
\begin{equation}
\label{2.16}
    S_n=N_n-\frac{a(\beta+1)}{\beta-a(\beta+1)}(a_n\mu_n)^{-1}M_n.
\end{equation}
{This writing is the key on which rely all of our analysis and our martingale approach.}
\par
Moreover, the asymptotic behavior of $(M_n)_{n\in\mathbb{N}}$ is closely related to $w_n$ defined in \eqref{2.10}. In fact, we have the following asymptotic result, which states the three regimes of the MARW.

\begin{lemma}
    In the diffusive regime when $p<\frac{4d\beta+2d+1}{4d(\beta+1)}$ or $a<1-\frac{1}{2(\beta+1)}$, we have
    \begin{equation}
    \label{2.17}
        \frac{w_n}{n^{1-2(a(\beta+1)-\beta)}}\to l(\beta)\quad \text{as}\quad n\to\infty
    \end{equation}
    with 
    \begin{equation*}
        l(\beta)=\frac{1}{1+2(\beta-a(\beta+1))}\bigg(\frac{\Gamma(a(\beta+1)+1)}{\Gamma(\beta+1)}\bigg)^2.
    \end{equation*}
    In the critical regime when $p=\frac{4d\beta+2d+1}{4d(\beta+1)}$ or $a=1-\frac{1}{2(\beta+1)}$, we have
    \begin{equation}
    \label{2.19}
        \frac{w_n}{\log n}\to\bigg(\frac{\Gamma(\beta+1+\frac{1}{2})}{\Gamma(\beta+1)}\bigg)^2\quad \text{as}\quad n\to\infty.
    \end{equation}
    In the superdiffusive regime when $p>\frac{4d\beta+2d+1}{4d(\beta+1)}$ or $a>1-\frac{1}{2(\beta+1)}$, we have
    \begin{equation}
    \label{2.20}
        w_n\to\sum\limits_{k=1}^\infty\bigg(\frac{\Gamma(a(\beta+1)+1)\Gamma(\beta+k)}{\Gamma(a(\beta+1)+k)\Gamma(\beta+1)}\bigg)^2<\infty\quad \text{as}\quad n\to\infty.
    \end{equation}
\end{lemma}

In order to investigate the asymptotic behavior of $(S_n)_{n\in\mathbb{N}}$, we first introduce an arbitrarily fixed test non-zero vector $u\in\mathbb{R}^d$ and we define
\begin{equation*}
    M_n(u)=u^TM_n\quad \;\text{and}\quad \;N_n(u)=u^TN_n\quad \text{for each}\quad n\in\mathbb{N}.
\end{equation*}
It is then clear that $(M_n(u))_{n\in\mathbb{N}}$ $(N_n(u))_{n\in\mathbb{N}}$ are real-valued locally square-integrable martingales for each fixed $u\in\mathbb{R}^d$. { We further infer that $(S_n(u))_{n\in\mathbb{N}}$ satisfies an equation analogous to \eqref{2.16}. In this setting, we have reduced the multidimensional martingales to real-valued martingales without loss of generality. This technique greatly simplifies our martingale analysis}.
From now on, we fix the test vector $u\in\mathbb{R}^d$ and we introduce the two-dimensional martingale $(\mathcal{L}_n(u))_{n\in\mathbb{N}}$ defined as
\begin{equation}
\label{3.2}
    \mathcal{L}_n(u)=
    \begin{pmatrix}
        N_n(u)\\M_n(u)
    \end{pmatrix}\quad \text{for each}\quad n\in\mathbb{N}.
\end{equation}
Denote the martingale increment $\epsilon_{n+1}=Y_{n+1}-\gamma_nY_n$ for each n. Then $(\epsilon_{n})_{n\in\mathbb{N}}$ satisfies the martingale difference relation $\mathbb{E}[\epsilon_{n+1}|\mathcal{F}_n]=0$. We obtain that
\begin{equation}\begin{aligned}
\label{3.3}
    \Delta\mathcal{L}_{n+1}(u)=\mathcal{L}_{n+1}(u)-\mathcal{L}_{n}(u)&=
    \begin{pmatrix}
        S_{n+1}(u)-S_n(u)+\frac{a(\beta+1)}{\beta-a(\beta+1)}\big(\mu_{n+1}^{-1}Y_{n+1}(u)-\mu_n^{-1}Y_n(u)\big)\\
        a_{n+1}Y_{n+1}(u)-a_{n}Y_{n}(u)
    \end{pmatrix}\\
    &=\begin{pmatrix}
        \frac{\beta\mu_{n+1}^{-1}}{\beta-a(\beta+1)}\big(\mu_{n+1}X_{n+1}(u)-(\gamma_n-1)Y_n(u)\big)\\
        a_{n+1}\epsilon_{n+1}(u)
    \end{pmatrix}\\
    &=\begin{pmatrix}
        \frac{\beta\mu_{n+1}^{-1}}{\beta-a(\beta+1)}\\a_{n+1}
    \end{pmatrix}\epsilon_{n+1}(u).
\end{aligned}\end{equation}\par



\section{Scaling limit and convergence}
\label{sec:chapter 3}
In this section, we discuss the scaling limit as well as the almost sure convergence in the diffusive, critical and the superdiffusive regimes, depending on the value of $p$ with respect to $(4d\beta+2d+1)/(4d(\beta+1))$. We also give the quadratic strong law in the diffusive regime as well as in the critical regime. Afterwards, the mean square convergence is established in the superdiffusive regime.
\subsection{The diffusive regime}

\begin{theorem}\label{thm: 4.1.1}
    We have the almost sure convergence
    \begin{equation*}
        \frac{1}{n}S_n\to0\quad \text{as}\quad n\to\infty\quad \mathbb{P}\text{-a.s.}
    \end{equation*}
\end{theorem}
\begin{proof}
    We have from \cite[Theorem 4.3.15]{Duflo} again that, for all $\gamma>0$,
    \begin{equation}
    \label{4.28}
          \frac{\norm{M_n}^2}{\lambda_{\text{max}}\langle M\rangle_n}  =o\big(\big(\log\Tr\langle M\rangle_n\big)^{1+\gamma}\big)\quad \mathbb{P}\text{-a.s.}
    \end{equation}
    From equation \eqref{3.33} and the fact that $\lambda_{\text{max}}\langle M\rangle_n\leq\Tr\langle M\rangle_n\leq w_n$, we get
    \begin{equation}
    \label{4.29}
        \norm{M_n}^2=o\big(w_n\big(\log w_n\big)^{1+\gamma}\big)\quad \mathbb{P}\text{-a.s.}
    \end{equation}
    By \eqref{2.17}, we observe
    \begin{equation*}
        \norm{M_n}^2=o\big(n^{1-2(a(\beta+1)-\beta)}\big(\log n\big)^{1+\gamma}\big)\quad \mathbb{P}\text{-a.s.}
    \end{equation*}
    Since $M_n=a_nY_n$, we have from equations \eqref{2.11} and \eqref{2.12}
    \begin{equation*}
        \frac{\norm{Y_n}^2}{(n\mu_{n+1})^2}=o\big(n^{-1}\big(\log n\big)^{1+\gamma}\big)\quad \mathbb{P}\text{-a.s.}
    \end{equation*}
    which implies
    \begin{equation*}
        \frac{Y_n}{n\mu_{n+1}}\to0\quad \text{as}\quad n\to\infty\quad \mathbb{P}\text{-a.s.}
    \end{equation*}
    By \eqref{3.34} and \cite[Theorem 4.3.15]{Duflo} again, we find that
    \begin{equation}
    \label{4.33}
        \norm{N_n}^2=o\big(n\big(\log n\big)^{1+\gamma}\big)\quad \mathbb{P}\text{-a.s.}
    \end{equation}
    Moreover, we obtain from equation \eqref{2.16}
    \begin{equation*}
        \frac{1}{n^2}\norm{S_n+\frac{a(\beta+1)}{(\beta-a(\beta+1))\mu_{n+1}}Y_n}^2=o\big(n^{-1}\big(\log n\big)^{1+\gamma}\big)\quad \mathbb{P}\text{-a.s.}
    \end{equation*}
    Hence, we conclude that
    \begin{equation*}
        \frac{S_n}{n}+\frac{a(\beta+1)}{\beta-a(\beta+1)}\cdot\frac{Y_n}{n\mu_{n+1}}\to0\quad \text{as}\quad n\to\infty\quad \mathbb{P}\text{-a.s.}
    \end{equation*}
    and the proof is complete.
\end{proof}

\begin{theorem}\label{thm: 4.1.2}
    We have the quadratic strong law
    \begin{equation*}
        \frac{1}{\log n}\sum\limits_{k=1}^n\frac{S_kS_k^T}{k^2}\to\frac{2\beta+1-a}{(1-a)(1-2(a(\beta+1)-\beta))}\cdot\frac{1}{d}I_d\quad \text{as}\quad n\to\infty\quad \mathbb{P}\text{-a.s.}
    \end{equation*}
\end{theorem}
\begin{proof}
    We will check that all the conditions of \cite[Theorem A.3]{Laulin2} are satisfied, see also \cite{Chaabane, Sheng}. The condition $(H.1)$ is satisfied thanks to Lemma \ref{lem:H1-D} while the condition $(H.2)$ directly follows from Lemma \ref{lem:H2-D} and the condition $(H.4)$ is exactly the statement of Lemma \ref{lem:H4-D}. Therefore,
    \begin{equation*}
        \frac{1}{\log(\det V_n^{-1})^2}\sum\limits_{k=1}^n\bigg(\frac{(\det V_k)^2-(\det V_{k+1})^2}{(\det V_k)^2}\bigg)V_k\mathcal{L}_k(u)\mathcal{L}_k(u)^TV_k^T\to\frac{1}{d}u^TuV_{t=1}
    \end{equation*}
    as $n\to\infty$ $\mathbb{P}$-a.s. On the one hand, we have from \eqref{4.23old} that
    \begin{equation}
    \label{4.38}
        \frac{1}{\log n} \sum\limits_{k=1}^n\bigg(\frac{(\det V_k)^2-(\det V_{k+1})^2}{(\det V_k)^2}\bigg)V_k\mathcal{L}_k(u)\mathcal{L}_k(u)^TV_k^T\to\frac{2(1-a)(\beta+1)}{d}u^TuV_{t=1}
    \end{equation}
    as $n\to\infty$ $\mathbb{P}$-a.s. On the other hand, by \eqref{2.11}, \eqref{2.12} and \eqref{4.23old}, we have
    \begin{equation*}
        n\bigg(\frac{(\det V_n)^2-(\det V_{n+1})^2}{(\det V_n)^2}\bigg)\to2(1-a)(\beta+1)\quad \text{as}\quad n\to\infty\quad \mathbb{P}\text{-a.s.}
    \end{equation*}
    Finally, we obtain from \eqref{4.2} and \eqref{4.38} that
    \begin{equation}
    \label{4.40}
        \frac{1}{\log n}\sum\limits_{k=1}^n\frac{u^TS_kS^T_ku}{k^2}=\frac{1}{\log n}\sum\limits_{k=1}^n\frac{v^TV_k\mathcal{L}_k(u)\mathcal{L}_k(u)^TV^T_kv}{k}\to v^TV_{t=1}v\cdot\frac{1}{d}u^Tu
    \end{equation}
    as $n\to\infty$ $\mathbb{P}$-a.s. Since $u\in\mathbb{R}^d$ is arbitrary, the assertion follows from \eqref{4.40}.
\end{proof}

\begin{theorem}\label{thm: 4.1.3}
    The MARW admits a scaling limit at the diffusive regime, or convergence in distribution, in the Skorokhod space $\mathfrak{D}(\mathbb{R}_+)$ of càdlàg functions, in the sense that
    \begin{equation*}
        \bigg(\frac{1}{\sqrt{n}}S_{\lfloor nt\rfloor},\;t\geq0\bigg)\Longrightarrow\bigg(W_t,\;t\geq0\bigg)
    \end{equation*}
    where $(W_t)_{t\geq0}$ is a continuous $\mathbb{R}^d$-valued centered Gaussian process such that $W_0=0$ and with covariance
    \begin{equation}\begin{aligned}
    \label{4.42}
        \mathbb{E}\big[W_sW^T_t\big]&=\frac{a(\beta+1)(1-a)+a\beta}{(2(\beta+1)(1-a)-1)(a-\beta(1-a))(1-a)}s\Big(\frac{t}{s}\Big)^{a-\beta(1-a)}\cdot\frac{1}{d}I_d\\
        &\quad \quad +\frac{\beta}{(\beta(1-a)-a)(1-a)}s\cdot\frac{1}{d}I_d\quad \text{for all}\quad 0\leq s\leq t<\infty.
    \end{aligned}\end{equation}
\end{theorem}
\begin{proof}
    We will check that all the conditions of \cite[Theorem A.2]{Laulin2} are satisfied, see also \cite{Chaabane, Sheng}. The condition $(H.1)$ is satisfied thanks to Lemma \ref{lem:H1-D} while the condition $(H.2)$ directly follows from Lemma \ref{lem:H2-D} and the condition $(H.3)$ is exactly the statement of Lemma \ref{lem:H3-D}. Consequently, we have the convergence in distribution in the Skorokhod space $\mathfrak{D}(\mathbb{R}_+)$ such that
    \begin{equation*}
        \bigg(V_n\mathcal{L}_{\lfloor nt\rfloor}(u),\;t\geq0\bigg)\Longrightarrow\bigg(\mathcal{W}_t(u),\;t\geq0\bigg)
    \end{equation*}
    where $(\mathcal{W}_t(u))_{t\geq0}$ is a continuous $\mathbb{R}^2$-valued centered Gaussian process such that $W_0=0$ and with covariance
    \begin{equation*}
        \mathbb{E}\big[\mathcal{W}_s(u)\mathcal{W}_t(u)^T\big]=\frac{1}{d}u^TuV_s\quad \text{for all}\quad 0\leq s\leq t<\infty.
    \end{equation*}
    From \eqref{2.11}, \eqref{2.12}, and \eqref{2.16}, we see that $S_{\lfloor nt\rfloor}(u)$ is asymptotically equivalent to
    \begin{equation*}
        N_{\lfloor nt\rfloor}(u)+t^{\beta-a(\beta+1)}\frac{a(\beta+1)}{\beta-a(\beta+1)}(a_n\mu_n)^{-1}M_{\lfloor nt\rfloor}(u)\quad \mathbb{P}\text{-a.s.}
    \end{equation*}
    Multiplying on the left side by $v_t=(1,t^{a(\beta+1)-\beta})^T$, we obtain
    \begin{equation*}
        \bigg(\frac{1}{\sqrt{n}}S_{\lfloor nt\rfloor}(u),\;t\geq0\bigg)\Longrightarrow\bigg(W_t(u),\;t\geq0\bigg)
    \end{equation*}
    with $W_t(u)=v_t^T\mathcal{W}_t(u)$. Hereafter, when $0\leq s\leq t<\infty$, we have the covariance
    \begin{equation}
    \label{4.47}
        \mathbb{E}\big[W_s(u)W_t(u)^T\big]=v_s^T\mathbb{E}\big[\mathcal{W}_s(u)\mathcal{W}_t(u)^T\big]v_t=\frac{1}{d}(u^Tu)v_s^TV_sv_t.
    \end{equation}
    Solving \eqref{4.47}, we have
    \begin{equation*}
        \mathbb{E}\big[W_sW_t^T\big]=\frac{1}{d}v_s^TV_sv_t\quad \text{for all}\quad 0\leq s\leq t<\infty
    \end{equation*}
    and the assertion \eqref{4.42} is verified.
\end{proof}

\subsection{The critical regime}

\begin{theorem}
\label{thm: 4.2.1}
    We have the almost sure convergence
    \begin{equation*}
        \frac{1}{\sqrt{n}\log n}S_n\to0\quad \text{as}\quad n\to\infty\quad \mathbb{P}\text{-a.s.}
    \end{equation*}
\end{theorem}
\begin{proof}
    We still have \eqref{4.28} and \eqref{4.29} such that
    \begin{equation*}
        \norm{M_n}^2=o\big(w_n\big(\log w_n\big)^{1+\gamma}\big)\quad \text{for all}\quad \gamma>0\quad \mathbb{P}\text{-a.s.}
    \end{equation*}
    However, in the critical regime, we have \eqref{2.19} rather than \eqref{2.17}, and
    \begin{equation*}
        \frac{w_n}{\log n}\to\bigg(\frac{\Gamma(\beta+1+\frac{1}{2})}{\Gamma(\beta+1)}\bigg)^2\quad \text{as}\quad n\to\infty.
    \end{equation*}
    Since \eqref{2.11}, \eqref{2.12}, and since $M_n=a_nY_n$, we observe for all $\gamma>0$ that
    \begin{equation*}
        \frac{\norm{Y_n}^2}{n(\log n)^2\mu_n^2}=o\big((\log n)^{-1}\big(\log\log n\big)^{1+\gamma}\big)\quad \mathbb{P}\text{-a.s.}
    \end{equation*}
    In this regard
    \begin{equation}
    \label{4.53}
        \frac{Y_n}{\sqrt{n}\log n\mu_n}\to0\quad \text{as}\quad n\to\infty\quad \mathbb{P}\text{-a.s.}
    \end{equation}
    Similarly, we still have \eqref{3.34} and
    \begin{equation*}
        \norm{N_n}^2=o\big(n\big(\log n\big)^{1+\gamma}\big)\quad \text{for all}\quad \gamma>0\quad \mathbb{P}\text{-a.s.}
    \end{equation*}
    Then 
    \begin{equation*}
        \frac{\norm{N_n}^2}{n(\log n)^2}=o\big((\log n)^{\gamma-1}\big)\quad \text{for all}\quad \gamma\in(0,1)\quad \mathbb{P}\text{-a.s.}
    \end{equation*}
    and therefore
    \begin{equation*}
        \frac{N_n}{\sqrt{n}\log n}\to0\quad \text{as}\quad n\to\infty\quad \mathbb{P}\text{-a.s.}
    \end{equation*}
    By \eqref{2.16}, we can hereafter conclude that
    \begin{equation*}
        \frac{S_n}{\sqrt{n}\log n}+\frac{a(\beta+1)}{\beta-a(\beta+1)}\cdot\frac{Y_n}{\sqrt{n}\log n\mu_n}\to0\quad \text{as}\quad n\to\infty\quad \mathbb{P}\text{-a.s.}
    \end{equation*}
    Combining with \eqref{4.53}, the assertion is verified.
\end{proof}

\begin{theorem}
    We have the quadratic strong law
    \begin{equation*}
        \frac{1}{\log\log n}\sum\limits_{k=1}^n\frac{S_kS_k^T}{(k\log k)^2}\to(2\beta+1)^2\cdot\frac{1}{d}I_d\quad \text{as}\quad n\to\infty\quad \mathbb{P}\text{-a.s.}
    \end{equation*}
\end{theorem}
\begin{proof}
     We will check that all the conditions of \cite[Theorem A.3]{Laulin2} are satisfied. The condition $(H.1)$ is satisfied thanks to Lemma \ref{lem:H1-C} while the condition $(H.2)$ directly follows from Lemma \ref{lem:H2-C} and the condition $(H.4)$ is exactly the statement of Lemma \ref{lem:H4-C}. Therefore,
    \begin{equation}
    \label{4.73}
        \frac{1}{\log(\det W_n^{-1})^2}\sum\limits_{k=1}^n\bigg(\frac{(\det W_k)^2-(\det W_{k+1})^2}{(\det W_k)^2}\bigg)W_k\mathcal{L}_k(u)\mathcal{L}_k(u)^TW_k^T\to\frac{1}{d}u^TuW
    \end{equation}
    as $n\to\infty$ $\mathbb{P}$-a.s. On the one hand, we have from \eqref{4.60}
    \begin{equation*}
        \frac{1}{\log\log n} \sum\limits_{k=1}^n\bigg(\frac{(\det W_k)^2-(\det W_{k+1})^2}{(\det W_k)^2}\bigg)W_k\mathcal{L}_k(u)\mathcal{L}_k(u)^TW_k^T\to\frac{1}{d}u^TuW
    \end{equation*}
    as $n\to\infty$ $\mathbb{P}$-a.s. On the other hand, by \eqref{2.11}, \eqref{2.12}, and \eqref{4.59}, we have
    \begin{equation*}
        n\log n\bigg(\frac{(\det W_k)^2-(\det W_{k+1})^2}{(\det W_k)^2}\bigg)\to(2\beta+1)^2\quad \text{as}\quad n\to\infty\quad \mathbb{P}\text{-a.s.}
    \end{equation*}
    Then, we obtain from \eqref{4.2} and \eqref{4.73} that
    \begin{equation}
        \label{4.76}
        \frac{1}{\log\log n}\sum\limits_{k=1}^n\frac{u^TS_kS^T_ku}{(k\log k)^2}=\frac{1}{\log\log n}\sum\limits_{k=1}^n\frac{w^TW_k\mathcal{L}_k(u)\mathcal{L}_k(u)^TW^T_kw}{k\log k}\to\frac{(2\beta+1)^2}{d}u^Tu
    \end{equation}
    as $n\to\infty$ $\mathbb{P}$-a.s. Since $u\in\mathbb{R}^d$ is arbitrary, the assertion follows from \eqref{4.76}.
\end{proof}

\begin{theorem}
\label{thm:4.2.3}
    The MARW admits a scaling limit at the critical regime, or convergence in distribution, in the Skorokhod space $\mathfrak{D}(\mathbb{R}_+)$ of càdlàg functions, in the sense that
    \begin{equation*}
        \bigg(\frac{1}{\sqrt{n^t\log n}}S_{\lfloor n^t\rfloor},\;t\geq0\bigg)\Longrightarrow\bigg((2\beta+1)B_t,\;t\geq0\bigg)
    \end{equation*}
    where $(B_t)_{t\geq0}$ is a continuous $d$-dimensional canonical Brownian motion with covariance
    \begin{equation*}
       {  \mathbb{E}\big[B_sB^T_t\big]=s\cdot\frac{1}{d}I_d}\quad \text{for all}\quad 0\leq s\leq t<\infty.
    \end{equation*}
\end{theorem}
\begin{proof}
    We will check that all the three conditions of \cite[Theorem A.2]{Laulin2} are satisfied, see also \cite[Theorem 1]{Touati}. First of all, by \eqref{2.19} and \eqref{3.43} we know that
    \begin{equation}
    \label{4.80}
        w_n^{-1/2}\langle M(u)\rangle_{\lfloor n^t\rfloor}w_n^{-1/2}\to \frac{t}{d}\cdot u^Tu \quad \text{as}\quad n\to\infty\quad \mathbb{P}\text{-a.s.} 
    \end{equation}
    Hence the condition $(H.1)$ is satisfied. Notice that
    \begin{equation}
    \label{4.81}
        \sum\limits_{k=1}^{\lfloor n^t\rfloor}\frac{1}{w_n}\mathbb{E}\big[\Delta M_k(u)^2\mathbbm{1}_{\{\abs{\Delta M_k(u)}\geq\epsilon\sqrt{w_k}\}}|\mathcal{F}_{k-1}\big]\leq \sum\limits_{k=1}^{\lfloor n^t\rfloor}\bigg(\frac{w_{\lfloor n^t\rfloor}}{w_n}\bigg)^2\frac{1}{\epsilon^2w_{\lfloor n^t\rfloor}^2}\mathbb{E}\big[\Delta M_k(u)^4|\mathcal{F}_{k-1}\big],
    \end{equation}
    since \eqref{2.11}, \eqref{2.12}, and \eqref{4.13}, we observe that
    \begin{equation}
    \label{4.82}
         \sum\limits_{k=1}^{\lfloor n^t\rfloor}\abs{\Delta M_k(u)^4}\leq C_1(\beta)\norm{u}^4\sum\limits_{k=1}^{\lfloor n^t\rfloor}(a_k\mu_k)^4\leq C_2(\beta)\norm{u}^4\sum\limits_{k=1}^{\lfloor n^t\rfloor}\frac{1}{k^2}\quad \mathbb{P}\text{-a.s.}
    \end{equation}
    with constants $C_1(\beta),C_2(\beta)>0$. Therefore, by \eqref{4.81} and \eqref{4.82}, we have
    \begin{equation*}
         \sum\limits_{k=1}^{\lfloor n^t\rfloor}\frac{1}{w_n}\mathbb{E}\big[\Delta M_k(u)^2\mathbbm{1}_{\{\abs{\Delta M_k(u)}\geq\epsilon\sqrt{w_k}\}}|\mathcal{F}_{k-1}\big]\leq C_3(\beta)\norm{u}^4\cdot\frac{t^2}{\epsilon^2}\cdot\frac{1}{n^t(\log n^t)^2}\quad \mathbb{P}\text{-a.s.}
    \end{equation*}
    Simplifying the above expression, we obtain
    \begin{equation}
    \label{4.84}
        \sum\limits_{k=1}^{\lfloor n^t\rfloor}\frac{1}{w_n}\mathbb{E}\big[\Delta M_k(u)^2\mathbbm{1}_{\{\abs{\Delta M_k(u)}\geq\epsilon\sqrt{w_k}\}}|\mathcal{F}_{k-1}\big]\to0\quad \text{as}\quad n\to\infty\quad \mathbb{P}\text{-a.s.}
    \end{equation}
    Then the condition $(H.2)$, or the Lindeberg condition, is satisfied by \eqref{4.84}. In this particular case at critical regime, \eqref{4.80} implies that the condition $(H.3)$ is satisfied. Hence
    \begin{equation*}
        \bigg(\frac{1}{\sqrt{w_n}}M_{\lfloor n^t\rfloor}(u),\;t\geq0\bigg)\Longrightarrow\bigg(\mathcal{B}_t(u),\;t\geq0\bigg)
    \end{equation*}
    where $(\mathcal{B}_t(u))_{t\geq0}$ is a continuous real-valued centered Gaussian process such that $B_0(u)=0$ and with covariance
    \begin{equation*}
        \mathbb{E}\big[\mathcal{B}_s(u)\mathcal{B}_t(u)\big]=\frac{s}{d}\cdot u^Tu\quad \text{for all}\quad 0\leq s\leq t<\infty.
    \end{equation*}
    In the critical regime, from \eqref{2.16} we can write
    \begin{equation}
    \label{4.87}
        S_{\lfloor n^t\rfloor}(u)=N_{\lfloor n^t\rfloor}(u)+(2\beta+1)\frac{M_{\lfloor n^t\rfloor}(u)}{a_{\lfloor n^t\rfloor}\mu_{\lfloor n^t\rfloor}}.
    \end{equation}
    From \eqref{3.44} we know that
    \begin{equation}
        \label{4.88}
        \frac{\langle N(u)\rangle_{\lfloor n^t\rfloor}}{n^t\log n}\to0\quad \;\text{and}\quad \;\frac{N_{\lfloor n^t\rfloor}(u)}{\sqrt{n^t\log n}}\to0\quad \text{as}\quad n\to\infty\quad \mathbb{P}\text{-a.s.}
    \end{equation}
    Using \eqref{2.11}, \eqref{2.12}, and \eqref{2.19} again, we conclude that
    \begin{equation*}
        \bigg(\frac{1}{\sqrt{n^t\log n}}S_{\lfloor n^t\rfloor}(u),\;t\geq0\bigg)\Longrightarrow\bigg((2\beta+1)B_t(u),\;t\geq0\bigg)
    \end{equation*}
    with 
    \begin{equation}
    \label{4.90}
      {    \mathbb{E}\big[B_s(u)B_t(u)\big]=s\cdot\frac{u^Tu}{d}}\quad \text{for all}\quad 0\leq s\leq t.
    \end{equation}
    Solving \eqref{4.90}, we get
    \begin{equation*}
        { \mathbb{E}\big[B_sB_t^T\big]=s\cdot\frac{1}{d}I_d}\quad \text{for all}\quad 0\leq s\leq t.
    \end{equation*}
    which completes the proof.
\end{proof}

\subsection{The superdiffusive regime}

\begin{theorem}
\label{thm:L-SD}
    We have the almost sure convergence
    \begin{equation*}
        \frac{1}{n^{a(\beta+1)-\beta}}S_n\to L_\beta\quad \text{as}\quad n\to\infty\quad \mathbb{P}\text{-a.s.}
    \end{equation*}
    where the limiting $L_\beta$ is an $\mathbb{R}^d$-valued random variable.
\end{theorem}
\begin{remark}
    In fact, from Theorem \ref{thm:mean-square-SD} below, we will see the random vector $L_\beta$ is non-degenerate.
\end{remark}
\begin{proof}
    From \eqref{2.20} and \eqref{3.43}, in the superdiffusive regime, we have
    \begin{equation*}
        \Tr\langle M\rangle_n\leq w_n\leq\sum\limits_{k=1}^\infty\bigg(\frac{\Gamma(a(\beta+1)+1)\Gamma(\beta+k)}{\Gamma(a(\beta+1)+k)\Gamma(\beta+1)}\bigg)^2<\infty\quad \text{for all}\quad n\in\mathbb{N}.
    \end{equation*}
    By \cite[Theorem 4.3.15]{Duflo}, this leads to
    \begin{equation*}
        M_n\to M\quad \text{as}\quad n\to\infty\quad \mathbb{P}\text{-a.s.}\quad \text{with}\quad M=\sum\limits_{k=1}^\infty a_k\epsilon_k.
    \end{equation*}
    By \eqref{2.13}, $M_n=a_nY_n$, and by \eqref{2.11}, we observe that
    \begin{equation}
    \label{4.95}
        \frac{Y_n}{n^{a(\beta+1)}}\to \frac{1}{\Gamma(a(\beta+1)+1)}M\quad \text{as}\quad n\to\infty\quad \mathbb{P}\text{-a.s.}
    \end{equation}
   Moreover, equations \eqref{4.33} still holds and, as $2a(\beta+1)>2\beta+1$ in the superdiffusive regime, we find that
    \begin{equation*}
        \frac{1}{n^{2(a(\beta+1)-\beta)}}\norm{S_n+\frac{a(\beta+1)}{(\beta-a(\beta+1))\mu_{n+1}}Y_n}^2=o\big(n^{-(1-2a(\beta+1)+2\beta)}\big(\log n\big)^{1+\gamma}\big)\quad \mathbb{P}\text{-a.s.}
    \end{equation*}
    Thanks to \eqref{2.12}, we obtain
    \begin{equation}
    \label{4.97}
        \frac{S_n}{n^{a(\beta+1)-\beta}}+\frac{a(\beta+1)}{\beta-a(\beta+1)}\cdot\frac{{ \Gamma(\beta+1)}Y_n}{n^{a(\beta+1)}}\to0\quad \text{as}\quad n\to\infty\quad \mathbb{P}\text{-a.s.}
    \end{equation}
    Combining \eqref{4.95}, it yields
    \begin{equation*}
        \frac{S_n}{n^{a(\beta+1)-\beta}}\to L_\beta\quad \text{as}\quad n\to\infty\quad \mathbb{P}\text{-a.s.}
    \end{equation*}
    where
    \begin{equation}
    \label{4.99}
        L_\beta=\frac{a(\beta+1)}{a(\beta+1)-\beta}\cdot\frac{{ \Gamma(\beta+1)}}{\Gamma(a(\beta+1)+1)}M
    \end{equation}
    and the assertion follows.
\end{proof}

\begin{theorem}
\label{thm:mean-square-SD}
    We have the following mean square convergence
    \begin{equation}
    \label{4.100}
        \mathbb{E}\bigg[\norm{\frac{1}{n^{a(\beta+1)-\beta}}S_n-L_\beta}^2\bigg]\to0\quad \text{as}\quad n\to\infty.
    \end{equation}
\end{theorem}
\begin{proof}
    For each test vector $u\in\mathbb{R}^d$, we have
    \begin{equation*}
        \mathbb{E}\big[M_n(u)^2\big]=\mathbb{E}\big[\langle M(u)\rangle_n\big]\leq \frac{1}{d} w_nu^Tu\quad \text{for all}\quad n\in\mathbb{N}.
    \end{equation*}
    From \eqref{2.20}, we obtain
    \begin{equation*}
        \sup\limits_{n\geq1}\mathbb{E}\big[M_n(u)^2\big]<\infty
    \end{equation*}
    which implies that $(M_n(u))_{n\in\mathbb{N}}$ is a martingale bounded in $L^2$. Therefore
    \begin{equation}
    \label{4.103}
        \mathbb{E}\big[\abs{M_n(u)-M(u)}^2\big]\to0\quad \text{as}\quad n\to\infty.
    \end{equation}
    Moreover, on the one hand \eqref{4.103} together with \eqref{4.95} implies that
    \begin{equation}
    \label{4.104}
        \mathbb{E}\bigg[\abs{\frac{1}{n^{a(\beta+1)}}Y_n(u)-Y(u)}^2\bigg]\to0\quad \text{as}\quad n\to\infty.
    \end{equation}
    On the other hand, from \eqref{3.44} we know that
    \begin{equation*}
        \mathbb{E}\big[N_n(u)^2\big]=\mathbb{E}\big[\langle N(u)\rangle_n\big]\leq\frac{1}{d}\bigg(\frac{\beta}{\beta-a(\beta+1)}\bigg)^2nu^Tu\quad \text{for all}\quad n\in\mathbb{N}.
    \end{equation*}
    Since $a(\beta+1)>\beta+\frac{1}{2}$ in the superdiffusive regime, we have
\begin{equation}
    \label{4.106}
        \mathbb{E}\bigg[\abs{\frac{1}{n^{a(\beta+1)-\beta}}N_n(u)}^2\bigg]\to0\quad \text{as}\quad n\to\infty.
    \end{equation}
   The proof is complete by combining \eqref{4.104} and \eqref{4.106}.
\end{proof}

\begin{remark}
\label{rmk:L-expectations-SD}
    The expected value of $L_\beta$ is
    \begin{equation}
    \label{4.107}
        \mathbb{E}\big[L_\beta \big]=0
    \end{equation}
    whereas its quadratic deviation is 
    \begin{equation}
    \label{4.108}
        \mathbb{E}\big[L_\beta L_\beta^T\big]=\bigg(\frac{a(\beta+1)}{\beta-a(\beta+1)}\bigg)^2\frac{\Gamma(\beta+1)^2\Gamma(2(a-1)(\beta+1)+1)}{\Gamma((2a-1)(\beta+1)+1)^2} \cdot\frac{1}{d}I_d.
    \end{equation}
\end{remark}


\begin{theorem}
\label{thm: 4.3.4}
    The MARW admits a scaling limit at the superdiffusive regime, or convergence in distribution, in the Skorokhod space $\mathfrak{D}(\mathbb{R}_+)$ of càdlàg functions, in the sense that
    \begin{equation}
    \label{4.112}
        \bigg(\frac{1}{n^{a(\beta+1)-\beta}}S_{\lfloor nt\rfloor},\;t\geq0\bigg)\Longrightarrow\bigg(Q_t,\;t\geq0\bigg)
    \end{equation}
    with the limiting $Q_t=t^{a(\beta+1)-\beta}L_\beta$ for all $t\geq0$.
\end{theorem}
\begin{proof}
    For all $t\geq0$ and from \eqref{4.97}, we observe that
    \begin{equation*}
        \frac{S_{\lfloor nt\rfloor}}{{{\lfloor nt\rfloor}}^{a(\beta+1)-\beta}}+\frac{a(\beta+1)}{\beta-a(\beta+1)}\cdot\frac{Y_{{\lfloor nt\rfloor}}}{{{\lfloor nt\rfloor}}^{a(\beta+1)}}\to0\quad \text{as}\quad n\to\infty\quad \mathbb{P}\text{-a.s.}
    \end{equation*}
    which implies
    \begin{equation}
    \label{4.114}
        \frac{1}{n^{a(\beta+1)-\beta}}S_n\to t^{a(\beta+1)-\beta}L_\beta\quad \text{as}\quad n\to\infty\quad \mathbb{P}\text{-a.s.}
    \end{equation}
    The $\mathbb{P}$-a.s. convergence in \eqref{4.114}holds in all finite-dimensional distributions which characterizes the Skorokhod space topology. Hence, we have \eqref{4.112} and the assertion is verified.
\end{proof}

\bigskip
\section{Scaling limit of the barycenter process}\label{sec: barycenter}
The study of the scaling limit of the MARW $(S_n)_{n\in\mathbb{N}}$ gives us some information on its asymptotic behavior. Nonetheless, to understand its pathwise geometric features, we need to discuss its barycenter, or center of mass process. Such topics have been raised and discussed in \cite{McRedmond/Wade, Wade/Xu}.
In this Section, we turn our attention to the above-mentioned barycenter process $(G_n)_{n\in\mathbb{N}}$ defined by
\begin{equation}
    \label{5.1}
    G_n\coloneqq\frac{1}{n}\sum\limits_{k=1}^n S_k 
\end{equation}
Our work contains the discussion on the scaling limit and the almost sure convergence in the diffusive, critical and superdiffusive regimes. The quadratic strong law in the diffusive and critical regimes is also discussed while the mean square convergence in the superdiffusive regime is established.
\subsection{Almost sure convergence}
The barycenter process was discussed in \cite{Bercu3} for the elephant random walk in dimension $d$, which is a special case of the process we study here when $\beta=0$. We first begin with the almost sure convergence. 

\begin{theorem}
    We have the almost sure convergence, in the diffusive regime,
    \begin{equation}
    \label{5.2}
        \frac{1}{n}G_n\to0\quad \text{as}\quad n\to\infty\quad \mathbb{P}\text{-a.s.}
    \end{equation}
    while in the critical regime,
    \begin{equation}
    \label{5.3}
        \frac{1}{\sqrt{n}\log n}G_n\to0\quad \text{as}\quad n\to\infty\quad \mathbb{P}\text{-a.s.}
    \end{equation}
    and, in the superdiffusive regime,
    \begin{equation}
    \label{5.4}
        \frac{1}{n^{a(\beta+1)-\beta}}G_n\to\frac{1}{1+a(\beta+1)-\beta}L_\beta\quad \text{as}\quad n\to\infty\quad \mathbb{P}\text{-a.s.}
    \end{equation}
     where $L_\beta$ was characterized in Theorems \ref{thm:L-SD} and \ref{rmk:L-expectations-SD}.
\end{theorem}
\begin{proof}
    In the diffusive regime, from \eqref{5.1} we observe that
    \begin{equation*}
        \frac{1}{n}G_n=\sum\limits_{k=1}^n\frac{k}{n^2}\cdot\frac{1}{k}S_k=\sum\limits_{k=1}^n\frac{1}{k}S_ka^\prime_{n,k}\quad \text{with}\quad a^\prime_{n,k}=\frac{k}{n^2}.
    \end{equation*}
    Since $\sum_{k=1}^na^\prime_{n,k}\leq1$ for all $n\in\mathbb{N}$ and the almost sure convergence in Theorem \ref{thm: 4.1.1}, from Lemma \ref{lem:toep1} we can conclude that
    \begin{equation*}
        \frac{1}{n}G_n=\sum\limits_{k=1}^n\frac{1}{k}S_ka^\prime_{n,k}\to0\quad \text{as}\quad n\to\infty\quad \mathbb{P}\text{-a.s.}
    \end{equation*}
    such that \eqref{5.2} is verified. In the critical regime, we have from \eqref{5.1} that
    \begin{equation*}
        \frac{1}{\sqrt{n}\log n}G_n=\frac{1}{n^{3/2}\log n}\sum\limits_{k=1}^nS_k=\sum\limits_{k=1}^n\frac{1}{\sqrt{k}\log k}S_ka^{\prime\prime}_{n,k}\quad \text{with}\quad a^{\prime\prime}_{n,k}=\frac{k^{1/2}\log k}{n^{3/2}\log n}.
    \end{equation*}
    Since $\sum_{k=1}^na^{\prime\prime}_{n,k}\leq1$ for all $n\in\mathbb{N}$ and the almost sure convergence in Theorem \ref{thm: 4.2.1} holds, we get from Lemma \ref{lem:toep1} hat
    \begin{equation*}
        \frac{1}{\sqrt{n}\log n}G_n=\sum\limits_{k=1}^n\frac{1}{\sqrt{k}\log k}S_ka^{\prime\prime}_{n,k}\to0\quad \text{as}\quad n\to\infty\quad \mathbb{P}\text{-a.s.}
    \end{equation*}
    and we obtain \eqref{5.3}. Finally, in the superdiffusive regime, we also get from \eqref{5.1} that
    \begin{equation*}
        \frac{1}{n^{a(\beta+1)-\beta}}G_n=\frac{1}{n^{1+a(\beta+1)-\beta}}\sum\limits_{k=1}^nS_n=\sum\limits_{k=1}^n\frac{1}{k^{a(\beta+1)-\beta}}S_ka^{\prime\prime\prime}_{n,k}\quad \text{with}\quad a^{\prime\prime\prime}_{n,k}=\frac{k^{a(\beta+1)-\beta}}{n^{1+a(\beta+1)-\beta}}.
    \end{equation*}
    Since
    \begin{equation*}
        \sum_{k=1}^na^{\prime\prime\prime}_{n,k}\to\frac{1}{1+a(\beta+1)-\beta}\quad \text{as}\quad n\to\infty
    \end{equation*}
    by a simple calculation, and because of the almost sure convergence in Theorem \ref{thm:L-SD}, we can conclude using Lemma \ref{lem:toep2}
    \begin{equation*}
        \frac{1}{n^{a(\beta+1)-\beta}}G_n\to\frac{1}{1+a(\beta+1)-\beta}L_\beta\quad \text{as}\quad n\to\infty\quad \mathbb{P}\text{-a.s.}
    \end{equation*}
   and \eqref{5.4} is verified.
\end{proof}
\subsection{Quadratic strong law}


\begin{theorem}
    In the diffusive regime, we have the quadratic strong law 
    \begin{equation*}
        \frac{1}{\log n}\sum\limits_{k=1}^n\frac{G_kG_k^T}{k^2}\to4I(a,\beta)\cdot\frac{1}{d}I_d\quad \text{as}\quad n\to\infty\quad \mathbb{P}\text{-a.s.}
    \end{equation*}
    where $I(a,\beta)$ is given explicitly
    \begin{equation*}
        I(a,\beta)=\frac{1}{\Gamma(a(\beta+1)+1)^2\Gamma(\beta+1)^2}\cdot\frac{2a^2(1-a)(\beta+1)^3}{3(\beta-a(\beta+1))^2(1-a(\beta+1)+\beta)}.
    \end{equation*}
\end{theorem}
\begin{proof}
    We will check that all the three conditions of \cite[Theorem A.2]{Laulin2} are satisfied. Looking back to \eqref{5.1}, we observe that
\begin{equation*}\begin{aligned}
    G_n&=\frac{1}{n}\sum\limits_{k=1}^nN_k-\frac{1}{n}\frac{a(\beta+1)}{\beta-a(\beta+1)}\sum\limits_{k=1}^n\frac{1}{a_k\mu_k}M_k
    =\frac{1}{n}\sum\limits_{k=1}^nN_k-\frac{1}{n}\frac{a(\beta+1)}{\beta-a(\beta+1)}\sum\limits_{k=1}^n\frac{1}{a_k\mu_k}\sum\limits_{l=1}^ka_l\epsilon_l.
\end{aligned}\end{equation*}
Then, changing the order of summation, we have
\begin{equation*}\begin{aligned}
    G_n&=\frac{1}{n}\sum\limits_{k=1}^nN_k-\frac{1}{n}\frac{a(\beta+1)}{\beta-a(\beta+1)}\sum\limits_{k=1}^na_k\epsilon_k\sum\limits_{l=k}^n\frac{1}{a_l\epsilon_l}\\
    &=\frac{1}{n}\sum\limits_{k=1}^nN_k-\frac{1}{n}\frac{a(\beta+1)}{\beta-a(\beta+1)}\sum\limits_{k=1}^na_k(\delta_n-\delta_{k-1})\epsilon_k
\end{aligned}\end{equation*}
where we define $\delta_n=\sum_{k=1}^n(a_k\mu_k)^{-1}$ for all $n\in\mathbb{N}$. Moreover, we denote
\begin{equation*}
    Z_n=\sum\limits_{k=1}^nN_k-\frac{a(\beta+1)}{\beta-a(\beta+1)}\sum\limits_{k=1}^na_k\delta_{k-1}\epsilon_k.
\end{equation*}
such that we have
\begin{equation*}
    G_n=\frac{1}{n}Z_n-\frac{\delta_n}{n}\cdot\frac{a(\beta+1)}{\beta-a(\beta+1)}\sum\limits_{k=1}^na_k\epsilon_k=\frac{1}{n}\bigg(Z_n-\frac{a(\beta+1)}{\beta-a(\beta+1)}\delta_nM_n\bigg).
\end{equation*}
For a fixed text vector $u\in\mathbb{R}^d$, we define
\begin{equation}
\label{eq:def-H}
    \mathcal{H}_n(u)=
    \begin{pmatrix}
        Z_n(u)\\M_n(u)
    \end{pmatrix}\quad \text{for all}\quad n\in\mathbb{N}.
\end{equation}
which implies
\begin{equation*}
    \Delta\mathcal{H}_n(u)=\mathcal{H}_{n+1}(u)-\mathcal{H}_n(u)=
    \begin{pmatrix}
        N_{n+1}(u)\epsilon_{n+1}(u)^{-1}-\frac{a(\beta+1)}{\beta-a(\beta+1)} a_{n+1}\delta_n\\ a_{n+1}
    \end{pmatrix}\epsilon_{n+1}(u).
\end{equation*}
Then, let
    \begin{equation*}
        V_n=\frac{1}{n^{3/2}}
        \begin{pmatrix}
            1&0\\0&\frac{a(\beta+1)}{\beta-a(\beta+1)}\delta_n
        \end{pmatrix}\quad \;\text{and}\quad \;v=
        \begin{pmatrix}
            1\\-1
        \end{pmatrix}.
    \end{equation*}
    Then it is immediate that
    \begin{equation}
    \label{5.24}
        v^TV_n\mathcal{H}_n(u)=\frac{1}{\sqrt{n}}G_n\quad \text{for all}\quad n\in\mathbb{N}
    \end{equation}
    and that
    \begin{equation*}\begin{aligned}
        &\lim\limits_{n\to\infty}V_n\langle\mathcal{H}(u)\rangle_nV_n^T=\lim\limits_{n\to\infty}\frac{1}{n^3}
        \begin{pmatrix}
            1&-1\\-1&1
        \end{pmatrix}\sum\limits_{k=1}^{n-1}\bigg(\frac{a(\beta+1)}{\beta-a(\beta+1)}\bigg)^2\delta_k^2a_{k+1}^2\mathbb{E}\big[\epsilon_{k+1}(u)^2|\mathcal{F}_k\big]\\
        &\quad \quad =\lim\limits_{n\to\infty}\frac{1}{n^3}\cdot\frac{a^2(1-a)(\beta+1)^3u^Tu}{d(\beta-a(\beta+1))^2(1-a(\beta+1)+\beta)}
        \begin{pmatrix}
            1&-1\\-1&1
        \end{pmatrix}\sum\limits_{k=1}^{n-1}\delta_k^2a_{k+1}^2\mu_{k+1}^2\quad \mathbb{P}\text{-a.s.}
    \end{aligned}\end{equation*}
    By \eqref{2.11} and \eqref{2.12}, we know that
    \begin{equation*}
        n^{-(1+a(\beta+1)-\beta)}\delta_n\to\frac{1}{1+a(\beta+1)-\beta}\cdot\frac{1}{\Gamma(a(\beta+1)+1)\Gamma(\beta+1)}\quad \text{as}\quad n\to\infty.
    \end{equation*}
    Hence the above calculation leads us to
    \begin{equation}
    \label{5.27}
        V_n\langle\mathcal{H}(u)\rangle_nV_n^T\to I(a,\beta)u^Tu\cdot\frac{1}{d}
        \begin{pmatrix}
            1&-1\\-1&1
        \end{pmatrix}\quad \text{as}\quad n\to\infty\quad \mathbb{P}\text{-a.s.}
    \end{equation}
    where
    \begin{equation}
    \label{5.28}
        I(a,\beta)=\frac{1}{1-2(a(\beta+1)-\beta)}\cdot\frac{a^2(1-a)(\beta+1)^3}{(\beta-a(\beta+1))^2(1-a(\beta+1)+\beta)}.
    \end{equation}
    Consequently, \eqref{5.27} ensures that the condition $(H.1)$ is satisfied. Notice that by \eqref{2.8} and \eqref{2.13}, there exists some constant $C_1(a,\beta)>0$ and similarly, by \eqref{2.11}, \eqref{2.12}, \eqref{4.14}, there exists some other constant $C_2(a,\beta)>0$ such that
    \begin{equation*}
        \norm{N_n}^2\leq C_1(a,\beta)n^2\quad \;\text{and}\quad \;a_k^2\epsilon_k(u)^2\leq C_2(a,\beta)n^2\delta_n^{-2}\quad \text{for all}\quad 1\leq k\leq n.
    \end{equation*}
    Moreover, notice that for all $1\leq k\leq n$,
    \begin{equation*}
        V_n\Delta\mathcal{H}_k(u)=\frac{1}{n^{3/2}}
        \begin{pmatrix}
            N_k(u)\epsilon_k(u)^{-1}-\frac{a(\beta+1)}{\beta-a(\beta+1)} a_k\delta_{k-1}\\
            \frac{a(\beta+1)}{\beta-a(\beta+1)} a_k\delta_n
        \end{pmatrix}\epsilon_k(u).
    \end{equation*}
    Hence, for all $1\leq k\leq n$, we observe that
    \begin{equation}
    \label{5.31}
        \norm{V_n\Delta\mathcal{H}_k(u)}^2\leq\frac{4a_k^2}{n^3}\bigg(\frac{a(\beta+1)}{\beta-a(\beta+1)}\bigg)^2\bigg(\bigg(\frac{\beta-a(\beta+1)}{a_ka(\beta+1)}\frac{N_k(u)}{\epsilon_k(u)}\bigg)^2+\delta_{k-1}^2+\delta_n^2\bigg)\epsilon_k(u)^2\leq\frac{C(a,\beta)}{n}
    \end{equation}
    for some constant $C(a,\beta)>0$. Consequently, we 
    \begin{equation*}
        \sum\limits_{k=1}^n\mathbb{E}\big[\norm{V_n\Delta\mathcal{H}_k(u)}^4\big]\leq\frac{1}{n}C(a,\beta)\to0\quad \text{as}\quad n\to\infty\quad \mathbb{P}\text{-a.s.}
    \end{equation*}
    since, for all $\epsilon>0$,
    \begin{equation}
    \label{5.33}
        \sum\limits_{k=1}^n\mathbb{E}\big[\norm{V_n\Delta\mathcal{H}_k(u)}^2\mathbbm{1}_{\{\norm{V_n\Delta\mathcal{H}_k(u)}>\epsilon\}}|\mathcal{F}_{k-1}\big]\leq\frac{1}{\epsilon^2}\sum\limits_{k=1}^n\mathbb{E}\big[\norm{V_n\Delta\mathcal{H}_k(u)}^4\big]\to0\quad \text{as}\quad n\to\infty\quad \mathbb{P}\text{-a.s.}
    \end{equation}
    Then the condition $(H.2)$, or the Lindeberg condition, is satisfied by \eqref{5.33}. Hereafter, by \eqref{2.11}, \eqref{2.12}, and by the definition of $\delta_n$, we know there exists some constant $C'(a,\beta)\neq0$ such that
    \begin{equation*}
        \frac{\log(\det V_n^{-1})^2}{\log n}\to C'(a,\beta)\quad \text{as}\quad n\to\infty.
    \end{equation*}
    This ensures that there exists some other constant $C''(a,\beta)>0$ such that
    \begin{equation*}
        \sum\limits_{n=1}^\infty\frac{1}{\big(\log(\det V_n^{-1})^2\big)^2}\mathbb{E}\big[\norm{V_n\Delta\mathcal{H}_n(u)}^4|\mathcal{F}_{n-1}\big]\leq C_2(a,\beta)\sum\limits_{n=1}^\infty\frac{1}{(\log n)^2}\mathbb{E}\big[\norm{V_n\Delta\mathcal{H}_n(u)}^4|\mathcal{F}_{n-1}\big].
    \end{equation*}
    Finally, using \eqref{5.31} leads to
    \begin{equation*}
        \sum\limits_{n=1}^\infty\frac{1}{(\log n)^2}\norm{V_n\Delta\mathcal{H}_n(u)}^4\leq C(a,\beta)\sum\limits_{n=1}^\infty\frac{1}{(n\log n)^2}<\infty\quad \mathbb{P}\text{-a.s.}
    \end{equation*}
    for some constant $C(a,\beta)>0$ depending only on $a$ and $\beta$. The condition $(H.4)$ is satisfied by combining the above with \eqref{5.33}. On the one hand,
    \begin{equation}
    \label{5.37}
        \frac{1}{\log(\det V_n^{-1})^2}\sum\limits_{k=1}^n\bigg(\frac{(\det V_k)^2-(\det V_{k+1})^2}{(\det V_k)^2}\bigg)V_k\mathcal{H}_k(u)\mathcal{H}_k(u)^TV_k^T\to\frac{1}{d}u^TuV
    \end{equation}
    as $n\to\infty$ $\mathbb{P}$-a.s. where
    \begin{equation*}
        V=
        \begin{pmatrix}
            1&-1\\-1&1
        \end{pmatrix}I(a,\beta)
    \end{equation*}
    and $I(a,\beta)$ has been specified in \eqref{5.28}. Then, we have
    \begin{equation*}
        \frac{1}{\log n} \sum\limits_{k=1}^n\bigg(\frac{(\det V_k)^2-(\det V_{k+1})^2}{(\det V_k)^2}\bigg)V_k\mathcal{H}_k(u)\mathcal{H}_k(u)^TV_k^T\to\frac{4-2(a(\beta+1)-\beta)}{d}u^TuV
    \end{equation*}
    as $n\to\infty$ $\mathbb{P}$-a.s. since
    \begin{equation*}
        \frac{\log n}{\log(\det V_n^{-1})^2}\to4-2(a(\beta+1)-\beta)\quad \text{as}\quad n\to\infty.
    \end{equation*}
    On the other hand, by \eqref{2.11} and \eqref{2.12}, we have
    \begin{equation*}
        n\bigg(\frac{(\det V_n)^2-(\det V_{n+1})^2}{(\det V_n)^2}\bigg)\to4-2\big(a(\beta+1)-\beta\big)\quad \text{as}\quad n\to\infty\quad \mathbb{P}\text{-a.s.}
    \end{equation*}
    Using \eqref{5.24} and \eqref{5.37}, we observe that
    \begin{equation*}
        \frac{1}{\log n}\sum\limits_{k=1}^n\frac{u^TG_kG^T_ku}{k^2}=\frac{1}{\log n}\sum\limits_{k=1}^n\frac{v^TV_k\mathcal{H}_k(u)\mathcal{H}_k(u)^TV^T_kv}{k}\to v^TVv\cdot\frac{1}{d}u^Tu
    \end{equation*}
    as $n\to\infty$ $\mathbb{P}$-a.s. Since $u\in\mathbb{R}^d$ is arbitrary, the assertion follows from \eqref{4.40}.
\end{proof}

\begin{theorem}
    In the critical regime, we have the quadratic strong law
    \begin{equation*}
        \frac{1}{\log\log n}\sum\limits_{k=1}^n\frac{G_kG_k^T}{(k\log k)^2}\to\frac{4(2\beta+1)^2}{9}\cdot\frac{1}{d}I_d\quad \text{as}\quad n\to\infty\quad \mathbb{P}\text{-a.s.}
    \end{equation*}
\end{theorem}
\begin{proof}
    We will check that all the three conditions of \cite[Theorem A.2]{Laulin2} are satisfied. Denote
    \begin{equation*}
        W_n=\frac{1}{n\sqrt{n\log n}}
        \begin{pmatrix}
            1&0\\0&\frac{a(\beta+1)}{\beta-a(\beta+1)}\delta_n
        \end{pmatrix}\quad \;\text{and}\quad \;
        w=
        \begin{pmatrix}
            1\\-1
        \end{pmatrix}.
    \end{equation*}
    Then, for $\mathcal{H}$ defined in \eqref{eq:def-H}, it is clear that
    \begin{equation*}
        w^TW_n\mathcal{H}_n(u)=\frac{1}{\sqrt{n\log n}}G_n\quad \text{for all}\quad n\in\mathbb{N}
    \end{equation*}
    and that
    \begin{equation*}\begin{aligned}
        &\lim\limits_{n\to\infty}W_n\langle\mathcal{H}(u)\rangle_nW_n^T=\lim\limits_{n\to\infty}\frac{1}{n^3\log n}
        \begin{pmatrix}
            1&-1\\-1&1
        \end{pmatrix}\sum\limits_{k=1}^{n-1}(2\beta+1)^2\delta_k^2a_{k+1}^2\mathbb{E}\big[\epsilon_{k+1}(u)^2|\mathcal{F}_k\big]\\
        &\quad \quad =\lim\limits_{n\to\infty}\frac{(2\beta+1)^2}{n^3\log n}\cdot\frac{u^Tu}{d}
        \begin{pmatrix}
            1&-1\\-1&1
        \end{pmatrix}\sum\limits_{k=1}^{n-1}\delta_k^2a_{k+1}^2\mu_{k+1}^2\quad \mathbb{P}\text{-a.s.}
    \end{aligned}\end{equation*}
    By \eqref{2.11} and \eqref{2.12}, we know that
    \begin{equation*}
        n^{-3/2}\delta_n\to\frac{2}{3}\cdot\frac{ \Gamma(\beta+1)}{\Gamma(\beta+1+\frac{1}{2})}\quad \text{as}\quad n\to\infty.
    \end{equation*}
    Hence, the above calculation leads us to
    \begin{equation}
    \label{5.49}
        W_n\langle\mathcal{H}(u)\rangle_nW_n^T\to I(\beta)u^Tu\cdot\frac{1}{d}
        \begin{pmatrix}
            1&-1\\-1&1
        \end{pmatrix}\quad \text{as}\quad n\to\infty\quad \mathbb{P}\text{-a.s.}\quad \text{with}\quad I(\beta)=\frac{4(2\beta+1)^2}{9}.
    \end{equation}
  Consequently, the condition $(H.1)$ is satisfied thanks to \eqref{5.49}. Notice that by \eqref{2.8} and \eqref{2.13}, there exists some constant $C_1(\beta)>0$ and similarly, there exists some constant $C_2(\beta)>0$ such that
    \begin{equation*}
        \norm{N_n}^2\leq C_1(\beta)n^2\quad \;\text{and}\quad \;a_k^2\epsilon_k(u)^2\leq C_2(\beta)n^2\delta_n^{-2}\log n\quad \text{for all}\quad 1\leq k\leq n.
    \end{equation*}
    Then, notice for all $1\leq k\leq n$ that
    \begin{equation*}
        W_n\Delta\mathcal{H}_k(u)=\frac{1}{n\sqrt{n\log n}}
        \begin{pmatrix}
            N_k(u)\epsilon_k(u)^{-1}-\frac{a(\beta+1)}{\beta-a(\beta+1)} a_k\delta_{k-1}\\
            \frac{a(\beta+1)}{\beta-a(\beta+1)} a_k\delta_n
        \end{pmatrix}\epsilon_k(u).
    \end{equation*}
    The ensures that, for all $1\leq k\leq n$, 
    \begin{equation}
    \label{5.52}
        \norm{W_n\Delta\mathcal{H}_k(u)}^2\leq\frac{4a_k^2}{n^3\log n} (2\beta+1)^2\bigg(\big((2\beta+1)^{-2}\frac{N_k(u)}{\epsilon_k(u)}\big)^2+\delta_{k-1}^2+\delta_n^2\bigg)\epsilon_k(u)^2\leq\frac{C(\beta)}{n}
    \end{equation}
    for some constant $C(\beta)>0$. Hence,
    \begin{equation*}
        \sum\limits_{k=1}^n\mathbb{E}\big[\norm{W_n\Delta\mathcal{H}_k(u)}^4\big]\leq\frac{1}{n}C(\beta)\to0\quad \text{as}\quad n\to\infty\quad \mathbb{P}\text{-a.s.}
    \end{equation*}
    since, for all $\epsilon>0$,
    \begin{equation}
    \label{5.54}
        \sum\limits_{k=1}^n\mathbb{E}\big[\norm{W_n\Delta\mathcal{H}_k(u)}^2\mathbbm{1}_{\{\norm{W_n\Delta\mathcal{H}_k(u)}>\epsilon\}}|\mathcal{F}_{k-1}\big]\leq\frac{1}{\epsilon^2}\sum\limits_{k=1}^n\mathbb{E}\big[\norm{W_n\Delta\mathcal{H}_k(u)}^4\big]\to0\quad \text{as}\quad n\to\infty.
    \end{equation}
    Therefore, the condition $(H.2)$, or the Lindeberg condition, is satisfied using \eqref{5.54}.
    Hereafter, we know that
    \begin{equation*}
        \frac{\log(\det W_n^{-1})^2}{\log\log n}\to 4\quad \text{as}\quad n\to\infty.
    \end{equation*}
    This ensures that there exists some constant $C_2(\beta)>0$ such that
    \begin{equation}
    \label{5.56}
        \sum\limits_{n=1}^\infty\frac{1}{\big(\log(\det W_n^{-1})^2\big)^2}\mathbb{E}\big[\norm{W_n\Delta\mathcal{H}_n(u)}^4|\mathcal{F}_{n-1}\big]\leq \sum\limits_{n=1}^\infty\frac{C_2(\beta)}{(\log\log n)^2}\mathbb{E}\big[\norm{W_n\Delta\mathcal{H}_n(u)}^4|\mathcal{F}_{n-1}\big].
    \end{equation}
    We get from \eqref{5.52} that
    \begin{equation*}
        \sum\limits_{n=1}^\infty\frac{1}{(\log\log n)^2}\norm{W_n\Delta\mathcal{H}_n(u)}^4\leq C(\beta)\sum\limits_{n=1}^\infty\frac{1}{(n\log n\log\log n)^2}<\infty\quad \mathbb{P}\text{-a.s.}
    \end{equation*}
    for some constant $C(\beta)>0$ depending only on$\beta$. The condition $(H.4)$ is satisfied using the above together with \eqref{5.56}. Then,
    \begin{equation*}
        \frac{1}{\log(\det W_n^{-1})^2}\sum\limits_{k=1}^n\bigg(\frac{(\det W_k)^2-(\det W_{k+1})^2}{(\det W_k)^2}\bigg)W_k\mathcal{H}_k(u)\mathcal{H}_k(u)^TW_k^T\to\frac{1}{d}u^TuW
    \end{equation*}
    as $n\to\infty$ $\mathbb{P}$-a.s. where
    \begin{equation*}
        W=\frac{4(2\beta+1)^2}{9}
        \begin{pmatrix}
            1&-1\\-1&1
        \end{pmatrix}.
    \end{equation*}
    Furthermore, on the one hand we have
    \begin{equation*}
        \frac{1}{\log\log n} \sum\limits_{k=1}^n\bigg(\frac{(\det W_k)^2-(\det W_{k+1})^2}{(\det W_k)^2}\bigg)W_k\mathcal{H}_k(u)\mathcal{H}_k(u)^TW_k^T\to\frac{1}{d}u^TuW
    \end{equation*}
    as $n\to\infty$ $\mathbb{P}$-a.s. since
    \begin{equation*}
        \frac{\log\log n}{\log(\det W_n^{-1})^2}\to\frac{1}{4}\quad \text{as}\quad n\to\infty.
    \end{equation*}
    On the other hand, we have
    \begin{equation*}
        n\log n\bigg(\frac{(\det W_n)^2-(\det W_{n+1})^2}{(\det W_n)^2}\bigg)\to1\quad \text{as}\quad n\to\infty\quad \mathbb{P}\text{-a.s.}
    \end{equation*}
    By \eqref{5.24} and \eqref{5.37}, we observe that
    \begin{equation}
    \label{5.63}
        \frac{1}{\log\log n}\sum\limits_{k=1}^n\frac{u^TG_kG^T_ku}{(k\log k)^2}=\frac{1}{\log\log n}\sum\limits_{k=1}^n\frac{w^TW_k\mathcal{H}_k(u)\mathcal{H}_k(u)^TW^T_kw}{4k\log k}\to w^TWw\cdot\frac{1}{4d}u^Tu
    \end{equation}
    as $n\to\infty$ $\mathbb{P}$-a.s. Since $u\in\mathbb{R}^d$ is arbitrary, the assertion follows from \eqref{5.63}.
\end{proof}

\begin{theorem}
    In the superdiffusive regime, we have the mean square convergence, given by
    \begin{equation}
    \label{5.64}
        \mathbb{E}\bigg[\norm{\frac{1}{n^{a(\beta+1)-\beta}}G_n-\frac{1}{1+a(\beta+1)-\beta}L_\beta}^2\bigg]\to0\quad \text{as}\quad n\to\infty.
    \end{equation}
\end{theorem}
\begin{proof}
    For all test vector $u\in\mathbb{R}^d$, it is immediate that
    \begin{equation}\begin{aligned}
    \label{5.65}
        \mathbb{E}&\bigg[\abs{\frac{1}{n^{a(\beta+1)-\beta}}G_n(u)-\frac{1}{1+a(\beta+1)-\beta}L_\beta(u)}^2\bigg]\leq2\mathbb{E}\bigg[\abs{\frac{1}{n^{1+a(\beta+1)-\beta}}Z_n(u)}^2\bigg]\\
        &\quad +2\mathbb{E}\bigg[\abs{\frac{1}{n^{1+a(\beta+1)-\beta}}\cdot\frac{a(\beta+1)}{a(\beta+1)-\beta}\delta_nM_n-\frac{1}{1+a(\beta+1)-\beta}L_\beta}^2\bigg].
    \end{aligned}\end{equation}
    By \eqref{4.99} and \eqref{5.27}, the second term converges to zero. Looking back to the first term in \eqref{5.65}, we observe
    \begin{equation}\begin{aligned}
    \label{5.66}
        \mathbb{E}&\bigg[\abs{\frac{1}{n^{1+a(\beta+1)-\beta}}Z_n(u)}^2\bigg]\leq\frac{4}{n^{1+2(a(\beta+1)-\beta)}}\sum\limits_{k=1}^n\mathbb{E}\big[N_k(u)^2\big]\\
        &\quad +\frac{4}{n^{1+2(a(\beta+1)-\beta)}}\bigg(\frac{a(\beta+1)}{a(\beta+1)-\beta}\bigg)^2\mathbb{E}\bigg[\abs{\sum\limits_{k=1}^na_k\delta_{k-1}\epsilon_k(u)}^2\bigg].
    \end{aligned}\end{equation}
    The first term in \eqref{5.66} converges to zero because $\mathbb{E}[N_k(u)]\leq(u^Tu)n$ for all $1\leq k\leq n$, and moreover, in the superdiffusive regime we have $a(\beta+1)>\beta+1/2$. The second term in \eqref{5.66} also converges to zero because 
    \begin{equation*}
        n^{-(1+a(\beta+1)-\beta)}\delta_n\to\frac{1}{1+a(\beta+1)-\beta}\cdot\frac{1}{\Gamma(1+a(\beta+1))\Gamma(\beta+1)}\quad \text{as}\quad n\to\infty.
    \end{equation*}
    Finally, using the above and that $M(u)=\sum_{k=1}^\infty a_k\epsilon_k(u)$ is bounded in $L^2$, the assertion follows.
\end{proof}

\subsection{Scaling limit}

\begin{theorem}\label{thm: 5.3.1}
    The barycenter process admits a scaling limit at the diffusive regime, or convergence in distribution, in the Skorokhod space $\mathfrak{D}([0,1])$ of càdlàg functions, such that
    \begin{equation*}
        \bigg(\frac{1}{\sqrt{n}}G_{\lfloor nt\rfloor},\;t\geq0\bigg)\Longrightarrow\bigg(\int\limits_{0}^1 W_{tr}\,dr,\;t\geq0\bigg)
    \end{equation*}
    where $(W_t)_{t\geq0}$ is a continuous $\mathbb{R}^d$-valued centered Gaussian process define in Theorem \ref{thm: 4.1.3} with its covariance defined in \eqref{4.42}. In particular,
    \begin{equation}\begin{aligned}
    \label{5.17}
        &\quad \quad \mathbb{E}\bigg[\bigg(\int\limits_0^1W_{sv}\,dv\bigg)\bigg(\int\limits_0^1W_{tu}\,du\bigg)^T\bigg]=\frac{\beta}{3(\beta(1-a)-a)(1-a)}s\cdot\frac{1}{d}I_d\\
        &+\frac{2(a(\beta+1)(1-a)+a\beta)}{3(2(\beta+1)(1-a)-1)(a-\beta(1-a))(1-a)(1+(1-a)(\beta+1))} t^{a-\beta(1-a)}s^{1-a+\beta(1-a)}\cdot\frac{1}{d}I_d
    \end{aligned}\end{equation}
    for all $0\leq s\leq t<\infty$.
\end{theorem}
\begin{proof}
    An easy calculation leads to 
    \begin{equation*}
        \lim_{n\to\infty}\frac{1}{\sqrt{n}}G_{\lfloor nt\rfloor}=\lim_{n\to\infty}\int\limits_0^1\frac{1}{\sqrt{n}}S_{\lfloor ntr\rfloor}\,dr\Longrightarrow\int\limits_{0}^1 W_{tr}\,dr
    \end{equation*}
    which ensures that $G_{\lfloor nt\rfloor}$ is a continuous function of $S_{\lfloor ntr\rfloor}$  in $\mathfrak{D}([0,1])$.  Then, the last convergence in law is due to the functional central limit Theorem \ref{thm: 4.1.3}, with $(W_t)_{t\geq0}$ defined there. Hence, the barycenter process $(G_n)_{n\in\mathbb{N}}$ admits a Gaussian scaling limit in the diffusive regime as well, with covariance 
    \begin{equation*}
        \mathbb{E}\bigg[\bigg(\int\limits_0^1W_{sv}\,dv\bigg)\bigg(\int\limits_0^1W_{tu}\,du\bigg)^T\bigg]=2\int\limits_0^1\int\limits_0^u\mathbb{E}\big[W_{sv}W_{tu}^T\big]\,dv\,du.
    \end{equation*}
    Using \eqref{4.42}, the formula \eqref{5.17} and the assertion follows.
\end{proof}

\begin{theorem}
    The barycenter process admits a scaling limit at the critical regime, or convergence in distribution, in the Skorokhod space $\mathfrak{D}([0,1])$ of càdlàg functions, such that
    \begin{equation*}
        \bigg(\frac{1}{\sqrt{n^t\log n}}G_{\lfloor n^t\rfloor},\;t\geq0\bigg)\Longrightarrow\bigg(\int\limits_{0}^1  (2\beta+1) B_{tr}\,dr,\;t\geq0\bigg)
    \end{equation*}
    where $(B_t)_{t\geq0}$ is a continuous $\mathbb{R}^d$-valued centered Gaussian process define in Theorem \ref{thm:4.2.3} with its covariance defined in \eqref{4.90}.
    \begin{proof}
        For each $r\in[0,1]$, \eqref{2.16} and \eqref{4.80} implies that
        \begin{equation*}
            \lim\limits_{n\to\infty}\frac{1}{\sqrt{n^t\log n}}\cdot\frac{M_{\lfloor n^tr\rfloor}(u)}{a_{\lfloor n^tr\rfloor}\mu_{\lfloor n^tr\rfloor}}=\lim\limits_{n\to\infty}\frac{1}{\sqrt{n^t\log n}}\big(n^tr(\log n+\tfrac{r}{t}\log r)\big)^{1/2}\mathcal{B}_{tr}(u)\quad \mathbb{P}\text{-a.s.}
        \end{equation*}
        for all $u\in\mathbb{R}^d$. Moreover, \eqref{4.88} yields
        \begin{equation*}
            \lim\limits_{n\to\infty}\frac{1}{\sqrt{n^t\log n}}N_{\lfloor n^tr\rfloor}(u)=\lim\limits_{n\to\infty}r^{1/2}\cdot\frac{1}{\sqrt{n^tr\log n}}N_{\lfloor n^tr\rfloor}(u)=0\quad \mathbb{P}\text{-a.s.}
        \end{equation*}
        for all $u\in\mathbb{R}^d$. By \eqref{4.87}, we have
        \begin{equation*}
            \bigg(\frac{1}{\sqrt{n^t\log n}}S_{\lfloor n^tr\rfloor}(u),\;t\geq0\bigg)\Longrightarrow\bigg((2\beta+1)B_{tr}(u),\;t\geq0\bigg)
        \end{equation*}
        for all $u\in\mathbb{R}^d$ and $r\in[0,1]$. Hence, we use again
        \begin{equation*}
        \lim_{n\to\infty}
            \frac{1}{\sqrt{n^t\log n}}G_{\lfloor n^t\rfloor}=\lim_{n\to\infty}\int\limits_0^1\frac{1}{\sqrt{n^t\log n}}S_{\lfloor n^tr\rfloor}\,dr\Longrightarrow\int\limits_{0}^1 (2\beta+1)B_{tr}\,dr
        \end{equation*}
        and the assertion is verified.
    \end{proof}
\end{theorem}


\begin{theorem}
    The barycenter process admits a scaling limit at the superdiffusive regime, or convergence in distribution, in the Skorokhod space $\mathfrak{D}([0,1])$ of càdlàg functions, such that
    \begin{equation*}
        \bigg(\frac{1}{n^{a(\beta+1)-\beta}}G_{\lfloor nt\rfloor},\;t\geq0\bigg)\Longrightarrow\bigg(\int\limits_0^1 Q_{tr}\,dr,\;t\geq0\bigg)
    \end{equation*}
    with the covariance specified in \eqref{5.22} and the limiting $L_\beta$ characterized in Theorem \ref{thm:mean-square-SD} and $Q_t=t^{a(\beta+1)-\beta}L_\beta$ characterized in Theorem \ref{thm: 4.3.4} for all $t\geq0$.
\end{theorem}
\begin{proof}
    Again, we find that
    \begin{equation*}
       \lim_{n\to\infty} \frac{1}{n^{a(\beta+1)-\beta}}G_{\lfloor nt\rfloor}=\int\limits_0^1\frac{1}{n^{a(\beta+1)-\beta}}S_{\lfloor ntr\rfloor}\,dr\Longrightarrow\int\limits_{0}^1 Q_{tr}\,dr
    \end{equation*}
    which ensures that $G_{\lfloor nt\rfloor}$ is a continuous function of $S_{\lfloor ntr\rfloor}$  in $\mathfrak{D}([0,1])$.  Then, the last convergence in law is due to the functional central limit Theorem \ref{thm: 4.3.4}. Hence the barycenter process $(G_n)_{n\in\mathbb{N}}$ admits a non-degenerate scaling limit in the superdiffusive regime as well, with covariance
    \begin{equation*}\begin{aligned}
    \label{5.22}
        \mathbb{E}&\bigg[\bigg(\int\limits_0^1Q_{sv}\,dv\bigg)\bigg(\int\limits_0^1Q_{tu}\,du\bigg)^T\bigg]=2\int\limits_0^1\int\limits_0^u\mathbb{E}\big[Q_{sv}Q_{tu}^T\big]\,dv\,du=\frac{t^{a(\beta+1)-\beta}s^{a(\beta+1)-\beta}}{(1+a(\beta+1)-\beta)^2}\mathbb{E}\big[L_{\beta}L_{\beta}^T\big]\\
        &=\frac{t^{a(\beta+1)-\beta}s^{a(\beta+1)-\beta}}{(1+a(\beta+1)-\beta)^2}\bigg(\frac{a(\beta+1)}{\beta-a(\beta+1)}\bigg)^2\frac{\Gamma(2(a-1)(\beta+1)+1)}{\Gamma((2a-1)(\beta+1)+1)^2} \cdot\frac{1}{d}I_d
    \end{aligned}\end{equation*}
    for all $0\leq s\leq t<\infty$.
\end{proof}


\section{Velocity of quadratic mean displacement}\label{sec: rate}
In this Section, we investigate the velocity of the mean square displacement of the MARW. This quantitative estimates give us the information on how fast the limit Theorems in Section \ref{sec:chapter 3} are carried on. Similar convergence velocities have been discussed in \cite{Fan,Hayashi}, where the authors analyzed the convergence velocity of the moments of a one-dimensional elephant random walk of all orders. In the superdiffusive regime, the convergence velocity was discussed in \cite{Bercu4}. Here, only the rate of quadratic moment convergence for the MARW in all of the three (diffusive,  critical, and superdiffusive) regimes are discussed. 
\par
Following the limit Theorems in Section \ref{sec:chapter 3}, we expect the asymptotic behavior of the mean square displacement is as follows,
\begin{equation}
    \label{6.1}
    \mathbb{E}\big[S_nS_n^T\big]\sim
    \begin{cases}
            n\cdot\frac{(a-2\beta)(1-a)(\beta+1)+\beta(a+1)}{(2(\beta+1)(1-a)-1)(a-\beta(1-a))(1-a)}\cdot\frac{1}{d}Id & \text{when}\quad a<1-\frac{1}{2(\beta+1)}\\
            n\log n\cdot(2\beta+1)^2\cdot\frac{1}{d}Id & \text{when}\quad a=1-\frac{1}{2(\beta+1)}\\
            n^{2(a(\beta+1)-\beta)}\cdot\bigg(\frac{a(\beta+1)}{\beta-a(\beta+1)}\bigg)^2\frac{\Gamma(2(a-1)(\beta+1)+1)}{\Gamma((2a-1)(\beta+1)+1)^2}\cdot\frac{1}{d}Id & \text{when}\quad a>1-\frac{1}{2(\beta+1)},
    \end{cases} 
\end{equation}
where the notation $\sim$ indicates two sequences $a_n\sim b_n$ if and only if $a_n/b_n\to1$ as $n\to\infty$.
\par
The aim of this Section is not only to show that the above estimates \eqref{6.1} are valid, but also to investigate the exact velocity of their convergence in the diffusive and critical regime.
\subsection{Diffusive regime}
\begin{theorem}
    For all $p<(4d\beta+2d+1)/4d(\beta+1)$, we have, as $n\to\infty$,
        \begin{equation*}\begin{aligned}
        \frac{1}{n}\mathbb{E}\big[S_nS_n^T\big]-\frac{(a-2\beta)(1-a)(\beta+1)+\beta(a+1)}{(2(\beta+1)(1-a)-1)(a-\beta(1-a))(1-a)}&\cdot\frac{1}{d}Id\\
         \sim-(C_1n^{-2(1-a)(\beta+1)}&+C_2n^{-1})\cdot\frac{1}{d}Id.\quad\quad\quad
    \end{aligned}\end{equation*}
\end{theorem}
\begin{proof}
    Take the vector $v=(1,-1)^T$ and $V_n\in\mathbb{R}^{2\times2}$ as in \eqref{4.1}. Then, $\tfrac{1}{\sqrt{n}}S_n(u)=v^TV_n\mathcal{L}_n(u)$, where $\mathcal{L}_n(u)=(N_n(u),M_n(u))^T$ is as in \eqref{3.2}. In particular,
    \begin{equation*}\begin{aligned}
        \frac{1}{n}u^T\mathbb{E}\big[S_nS_n^T\big]u&=v^TV_n\mathbb{E}\big[\mathcal{L}_n(u)\mathcal{L}_n(u)^T\big]V_n^Tv\\
        &=v^TV_n\mathbb{E}\bigg[
        \begin{pmatrix}
            \mathbb{E}\big[N_n(u)^2\big] & \mathbb{E}\big[N_n(u)M_n(u)\big] \\
            \mathbb{E}\big[M_n(u)N_n(u)\big] &
            \mathbb{E}\big[M_n(u)^2\big]
        \end{pmatrix}
        \bigg]V_n^Tv\\
        &=v^TV_n\mathbb{E}\bigg[
        \begin{pmatrix}
            \mathbb{E}\big[\langle N(u)\rangle_n\big] & \mathbb{E}\big[\langle N(u),M(u)\rangle_n\big] \\
            \mathbb{E}\big[\langle M(u),N(u)\rangle_n\big] &
            \mathbb{E}\big[\langle M(u)\rangle_n\big]
        \end{pmatrix}
        \bigg]V_n^Tv.
    \end{aligned}\end{equation*}
    Therefore, 
    \begin{equation*}\begin{aligned}
        \frac{1}{n}u^T\mathbb{E}\big[S_nS_n^T\big]u&=\frac{1}{n}\mathbb{E}\big[\langle N(u)\rangle_n\big]+\frac{1}{na_n^2\mu_n^2}\bigg(\frac{a(\beta+1)}{\beta-a(\beta+1)}\bigg)^2\mathbb{E}\big[\langle M(u)\rangle_n\big]\\
        &\quad \quad \quad -\frac{2}{na_n\mu_n}\bigg(\frac{a(\beta+1)}{\beta-a(\beta+1)}\bigg)\mathbb{E}\big[\langle M(u),N(u)\rangle_n\big].
    \end{aligned}\end{equation*}
    Since the test vector $u\in\mathbb{R}^d$ is taken arbitrarily, we get from Lemmas \ref{lem: 6.0.2} and \ref{lem: 6.0.3} that
    \begin{equation*}\begin{aligned}
        &\frac{1}{n}\mathbb{E}\big[S_nS_n^T\big]-\frac{(a-2\beta)(1-a)(\beta+1)+\beta(a+1)}{(2(\beta+1)(1-a)-1)(a-\beta(1-a))(1-a)}\cdot\frac{1}{d}Id\\
        &\quad \quad \quad \quad \quad \quad \quad \quad \quad \quad \quad \quad \quad \quad \sim-(C_1n^{-2(1-a)(\beta+1)}+C_2n^{-1})\cdot\frac{1}{d}Id\quad \text{as}\quad n\to\infty.
    \end{aligned}\end{equation*}
\end{proof}
\subsection{Critical regime}
\begin{theorem}
    When $p=(4d\beta+2d+1)/4d(\beta+1)$, we have, as $n\to\infty$,
    \begin{equation*}
        \frac{1}{n\log n}\mathbb{E}\big[S_nS_n^T\big]-{(2\beta+1)^2}\cdot\frac{1}{d}Id\sim-(C_1(\log n)^{-1}+C_2n^{-1})\cdot\frac{1}{d}Id.
    \end{equation*}
\end{theorem}
\begin{proof}
    Take $w=(1,-1)^T$ and $W_n\in\mathbb{R}^{2\times2}$ as in \eqref{4.49}. Then $\frac{1}{\sqrt{n\log n}}S_n(u)=w^TW_n\mathcal{L}_n(u)$ as in \eqref{4.50} for all $u\in\mathbb{R}^d$. In particular,
    \begin{equation*}
        \frac{1}{n\log n}u^T\mathbb{E}\big[S_nS_n^T\big]u=w^TW_n\mathbb{E}\big[\mathcal{L}_n(u)\mathcal{L}_n(u)^T\big]W^T_nw.
    \end{equation*}
    Hence,
    \begin{equation*}
        \frac{1}{n\log n}u^T\mathbb{E}\big[S_nS_n^T\big]u=w^TW_n\mathbb{E}\bigg[
        \begin{pmatrix}
            \mathbb{E}\big[\langle N(u)\rangle_n\big] & \mathbb{E}\big[\langle N(u),M(u)\rangle_n\big] \\
            \mathbb{E}\big[\langle M(u),N(u)\rangle_n\big] &
            \mathbb{E}\big[\langle M(u)\rangle_n\big]
        \end{pmatrix}
        \bigg]W_n^Tw.
    \end{equation*}
    Therefore, we get by \eqref{2.19} as $n\to\infty$,
    \begin{equation*}
        \frac{1}{n\log n}u^T\mathbb{E}\big[S_nS_n^T\big]u=\frac{1}{n\log n}\bigg(\mathbb{E}\big[\langle N(u)\rangle_n\big]+\frac{(2\beta+1)^2}{a_n^2\mu_n^2}\mathbb{E}\big[\langle M(u)\rangle_n\big]\bigg),
    \end{equation*}
    which implies
    \begin{equation*}
        \frac{1}{n\log n}\mathbb{E}\big[S_nS_n^T\big]-{(2\beta+1)^2}\cdot\frac{1}{d}Id\sim-(C_1(\log n)^{-1}+C_2n^{-1})\cdot\frac{1}{d}Id\quad \text{as}\quad n\to\infty.
    \end{equation*}
\end{proof}
\subsection{Superdiffusive regime}
\begin{theorem}
    When $p>(4d\beta+2d+1)/4d(\beta+1)$, we have, as $n\to\infty$,
    { \begin{equation*}\begin{aligned}
        \frac{1}{n^{2(a(\beta+1)-\beta)}}\mathbb{E}\big[S_nS_n^T\big]-&\bigg(\frac{a(\beta+1)}{\beta-a(\beta+1)}\bigg)^2\frac{\Gamma(2(a-1)(\beta+1)+1)}{\Gamma((2a-1)(\beta+1)+1)^2}\cdot\frac{1}{d}Id\\
        &\quad \quad \quad \sim -(C_1n^{-4(a(\beta+1)-\beta)+1}+C_2n^{-2(a(\beta+1)-\beta)}).
    \end{aligned}\end{equation*}}
\end{theorem}
\begin{proof}
    Similar to previous computations for the diffusive regime, we have for all $u\in\mathbb{R}^d$,
    \begin{equation*}\begin{aligned}
        \frac{1}{n^{2(a(\beta+1)-\beta)}}u^T\mathbb{E}&\big[S_nS_n^T\big]u=\frac{1}{n^{2(a(\beta+1)-\beta)}}\mathbb{E}\big[\langle N(u)\rangle_n\big]\\
        &+\frac{1}{n^{2(a(\beta+1)-\beta)} a_n^2\mu_n^2}\bigg(\frac{a(\beta+1)}{\beta-a(\beta+1)}\bigg)^2\mathbb{E}\big[\langle M(u)\rangle_n\big]\\
        &-\frac{2}{n^{2(a(\beta+1)-\beta)} a_n\mu_n}\bigg(\frac{a(\beta+1)}{\beta-a(\beta+1)}\bigg)\mathbb{E}\big[\langle M(u),N(u)\rangle_n\big].
    \end{aligned}\end{equation*}
    Hence, by \eqref{2.11}, \eqref{2.12}, \eqref{2.20} and since $u\in\mathbb{R}^d$ is arbitrary,
    \begin{equation*}\begin{aligned}
        \frac{1}{n^{2(a(\beta+1)-\beta)}}\mathbb{E}\big[S_nS_n^T\big]-&\bigg(\frac{a(\beta+1)}{\beta-a(\beta+1)}\bigg)^2\frac{\Gamma(2(a-1)(\beta+1)+1)}{\Gamma((2a-1)(\beta+1)+1)^2}\cdot\frac{1}{d}Id\\
        &\quad \quad \quad \sim -(C_1n^{-4(a(\beta+1)-\beta)+1}+C_2n^{-2(a(\beta+1)-\beta)})\quad \text{as}\quad n\to\infty.
    \end{aligned}\end{equation*}
\end{proof}

\section{Cramér moderate deviations}\label{sec: Cramér moderate deviations}

In this Section, we discuss the Cramér moderate deviations for the multidimensional reinforced random walk $(S_n)_{n\in\mathbb{N}}$. The similar statistical quantity as well as the Berry-Esseen bound for the one-dimensional elephant random walk (ERW) without amnesia-reinforcement has been given in \cite{Fan}. Our derivation of Cramér moderate deviations for the MARW does not rely on a Berry-Esseen bound. The discussion of such statistical quantities is expected to reveal the transience property and the central limit Theorems for the MARW. For this direction, readers are refereed to \cite{Baur,Coletti}. Thanks to Lemma \ref{lem: 7.0.5} and Lemma \ref{lem: 7.0.6}, we can properly state the Cramér moderate deviations principles for the MARW.
\begin{theorem}\label{thm: 7.0.1}
    In the diffusive and critical regimes, we have the following Cramér moderate deviations for the MARW. Let $(\vartheta_n)_{n\in\mathbb{N}}\subseteq\mathbb{R}$ be a non-decreasing sequence so that $\vartheta_n/\sqrt{n}\to0$ as $n\to\infty$. Take any non-empty Borel set $B\subseteq\mathbb{R}^d$, then we have 
    \begin{equation*}\begin{aligned}
        -\inf\limits_{x\in\text{int}\,B}\frac{1}{2}\norm{x}^2&\leq\liminf\limits_{n\to\infty}\vartheta_n^{-2}\log\mathbb{P}\bigg(\frac{a_n\mu_n S_n}{\vartheta_n\sqrt{w_n}}\in B\bigg)\\
        &\leq\limsup\limits_{n\to\infty}\vartheta_n^{-2}\log\mathbb{P}\bigg(\frac{a_n\mu_n S_n}{\vartheta_n\sqrt{w_n}}\in B\bigg)\leq-\inf\limits_{x\in\text{cl}\,B}\frac{1}{2}\norm{x}^2,
    \end{aligned}\end{equation*}
    where $\text{int}\,B$ and $\text{cl}\,B$ denote the interior and the closure of $B\subseteq\mathbb{R}^d$, respectively.
\end{theorem}
\begin{proof}
    Our proof will only present the Cramér moderate deviations for the MARW in the diffusive regime. The same property for the critical regime follows from exactly the same steps. First, take $x_B=\inf_{x\in B}\norm{x}$. Then it is obvious that $\inf_{x\in\text{cl}\,B}\norm{x}\leq x_B$ and $\inf_{x\in\text{cl}\,B}\norm{x}^2/2\leq x_B^2/2$. Henceforth,
    \begin{equation}\label{7.14}
        \mathbb{P}\bigg(\frac{a_n\mu_n S_n}{\vartheta_n\sqrt{w_n}}\in B\bigg)\leq\sum\limits_{j=1}^d\mathbb{P}\bigg(\bigg|\frac{a_n\mu_nS_n^j}{\sqrt{w_n}}\bigg|\geq\frac{\vartheta_nx_B}{d}\bigg)\leq\big(1-\Phi(\vartheta_nx_B)\big)\mathscr{F}(B,\vartheta,n),
    \end{equation}
    where we write
    \begin{equation*}\begin{aligned}
        &\quad \quad \quad \mathscr{F}(B,\vartheta,n)\coloneqq2Cd\cdot\exp(\tfrac{1}{\sqrt{n}}\big(\tfrac{\vartheta_nx_B}{2d}\big)^3+\tfrac{1}{n}\big(\tfrac{\vartheta_nx_B}{2d}\big)^2+\tfrac{1}{\sqrt{n}}(1+\tfrac{1}{2}\log n)(1+\tfrac{\vartheta_nx_B}{2d}))\\
        &+2Cd\cdot\exp(\tfrac{1}{\sqrt{n}}\big(\tfrac{\vartheta_nx_B}{2d}\big)^3+\tfrac{1}{n^{2(1-a)(\beta+1)}}\big(\tfrac{\vartheta_nx_B}{2d}\big)^2+\tfrac{1}{\sqrt{n}}(1+\tfrac{1}{2}\log n)(n^{1/2-(1-a)(\beta+1)}+\tfrac{\vartheta_nx_B}{2d})).
    \end{aligned}\end{equation*}
    Hence,
    \begin{equation*}
        \limsup\limits_{n\to\infty}\vartheta^{-2}_n\log\mathbb{P}\bigg(\frac{a_n\mu_nS_n}{\vartheta_n\sqrt{w_n}}\in B\bigg)\leq-\frac{1}{2}x_B^2\leq-\inf\limits_{x\in\text{cl}\,B}\frac{1}{2}\norm{x}^2.
    \end{equation*}
    To achieve the asymptotic lower bound, we first notice that this assertion automatically holds if $\text{int}\,B=\emptyset$, whence $-\inf_{x\in\emptyset}\norm{x}^2/2=-\infty$. Consequently, we assume that $\text{int}\,B\neq\emptyset$. Notice that $\text{int}\,B$ is open in $\mathbb{R}^d$. Hence, for all $\epsilon_*>0$ sufficiently small, we could find $x_*\in\text{int}\,B$ with
    \begin{equation*}
        0<\frac{1}{2}\norm{x_*}^2<\inf\limits_{x\in\text{int}\,B}\frac{1}{2}\norm{x}^2+\epsilon_*\quad \;\text{and}\quad \;0<\min\big\{\abs{x_*^j}:\,1\leq j\leq d\big\}.
    \end{equation*}
    Choose $\epsilon_{**}$ sufficient small such that $0<\epsilon_{**}<\abs{x_*^j}$ for each $j=1,\ldots,d$. Then,
    \begin{equation*}
        U(x_*,\epsilon_{**})\subseteq\text{int}\,B\subseteq B,\quad \;\text{where}\quad U(x_*,\epsilon_{**})\coloneqq\big\{x\in\mathbb{R}^d:\,\abs{x^j-x_*^j}<\epsilon_{**}\;\,\text{for all}\;\,j\big\}.
    \end{equation*}
    On the other hand,
    \begin{equation*}\begin{aligned}
        \mathbb{P}\bigg(\frac{a_n\mu_n S_n}{\vartheta_n\sqrt{w_n}}\in B\bigg)&\geq\mathbb{P}\bigg(\frac{a_n\mu_n S_n}{\sqrt{w_n}}\in\vartheta_n\cdot U(x_*,\epsilon_{**})\bigg)\\
        &\geq\prod\limits_{j=1}^d\mathbb{P}\bigg(\vartheta_n(x_*^j+\epsilon_{**})\geq\frac{a_n\mu_n S_n^j}{\sqrt{w_n}} \geq\vartheta_n(x_*^j-\epsilon_{**})\bigg).
    \end{aligned}\end{equation*}
    From Lemma \ref{lem: 7.0.5} and Lemma \ref{lem: 7.0.6}, we know that
    \begin{equation*}
        \lim\limits_{n\to\infty}\mathbb{P}\bigg(\frac{a_n\mu_n S_n^j}{\sqrt{w_n}} \geq\vartheta_n(x_*^j+\epsilon_{**})\bigg)\bigg/\mathbb{P}\bigg(\frac{a_n\mu_n S_n^j}{\sqrt{w_n}} \geq\vartheta_n(x_*^j-\epsilon_{**})\bigg)=0\quad \text{for each}\quad j.
    \end{equation*}
    Similar to \eqref{7.14}, 
    \begin{equation*}
        \liminf\limits_{n\to\infty}\vartheta^{-2}_n\log\mathbb{P}\bigg(\frac{a_n\mu_nS_n}{\vartheta_n\sqrt{w_n}}\in B\bigg)\geq-\frac{1}{2}\norm{x_*-\epsilon_{**}}^2.
    \end{equation*}
    Letting $\epsilon_{**}\to0$, we observe that
    \begin{equation*}
         \liminf\limits_{n\to\infty}\vartheta^{-2}_n\log\mathbb{P}\bigg(\frac{a_n\mu_nS_n}{\vartheta_n\sqrt{w_n}}\in B\bigg)\geq-\frac{1}{2}\norm{x_*}^2\geq-\inf\limits_{x\in\text{int}\,B}\frac{1}{2}\norm{x}^2-\epsilon_*.
    \end{equation*}
    Since $\epsilon_*>0$ was take arbitrarily, letting $\epsilon_*\to0$, we verify the assertion.
\end{proof}

\bigskip


\appendix
\section{Technical Lemmas}
\label{appendix}
\subsection{Asymptotics of the processes}
We start by introducing  the following processes that are of great influence on the behavior of the random walk. Let $(e_1,e_2,\ldots,e_d)$ denote a canonical Euclidean basis of $\mathbb{R}^d$. For each $n\in\mathbb{N}$ and $1\leq j\leq d$, define
\begin{equation}
\label{3.4}
    N^X_n(j)=\sum\limits_{k=1}^n\mathbbm{1}_{\{X_k^j\neq0\}}\mu_k\quad \;\text{and}\quad \;\Sigma_n=\sum\limits_{j=1}^dN_n^X(j)e_je_j^T,
\end{equation}
such that $(\Sigma_n)_{n\in\mathbb{N}}$ is a matrix-valued process.
\begin{lemma}
\label{LEM:wn-regimes}
    We have the following almost sure convergence in the three regimes.
    \begin{equation}
    \label{3.16new}
        \frac{1}{n\mu_{n+1}}\Sigma_n\to\frac{1}{d(\beta+1)}I_d\quad \text{as}\quad n\to\infty\quad \mathbb{P}\text{-a.s.}
    \end{equation}
\end{lemma}
\begin{proof}
    For each $n\in\mathbb{N}$ and $1\leq j\leq d$, define
    \begin{equation}
    \label{3.16}
        \Lambda^X_n(j)=\frac{N^X_n(j)}{n}.
    \end{equation}
    It follows from \eqref{3.4} that
    \begin{equation*}
        \Lambda^X_{n+1}(j)=\frac{n}{n+1}\Lambda^X_{n}(j)+\frac{1}{n+1}\mathbbm{1}_{\{X_{n+1}^j\neq0\}}\mu_{n+1}.
    \end{equation*}
    Moreover, we observe thanks to \eqref{3.10} that
    \begin{equation*}\begin{aligned}
        \Lambda^X_{n+1}(j)&=\frac{n}{n+1}\cdot\gamma_n\Lambda^X_{n}(j)+\frac{1}{n+1}\mathbbm{1}_{\{X_{n+1}^j\neq0\}}\mu_{n+1}-\frac{a(\beta+1)}{n+1}\Lambda^X_n(j)\\
        &=\frac{n}{n+1}\cdot\gamma_n\Lambda^X_{n}(j)+\frac{\mu_{n+1}}{n}\delta^X_{n+1}(j)+\frac{(1-a)\mu_{n+1}}{d(n+1)}
    \end{aligned}\end{equation*}
    with
    \begin{equation*}
        \delta^X_{n+1}(j)=\mathbbm{1}_{\{X_{n+1}^j\neq0\}}- \mathbb{P}\big(X_{n+1}^j\neq0|\mathcal{F}_n\big).
    \end{equation*}
    Then, by \eqref{2.10} we know
    \begin{equation}
    \label{3.20}
        \Lambda^X_n(j)=\frac{1}{na_n}\bigg(\Lambda^X_1(j)+\frac{1-a}{d}\sum\limits_{k=2}^na_k\mu_k+H^X_n(j)\bigg)
    \end{equation}
    with
    \begin{equation*}
        H^X_n(j)=\sum\limits_{k=2}^na_k\mu_k\delta^X_k(j).
    \end{equation*}
    It is clear that for a fixed $1\leq j\leq d$, the real-valued process $(H^X_n(j))_{n\in\mathbb{N}}$ is locally square-integrable since it is a finite sum. Afterwards, this process appears to be a martingale adapted to $(\mathcal{F}_n)_{n\in\mathbb{N}}$ because $(\delta^X_{n}(j))_{n\in\mathbb{N}}$ satisfied the martingale difference relation $\mathbb{E}[\delta^X_{n+1}(j)|\mathcal{F}_n]=0$. It is obvious that
    \begin{equation*}
        \langle H^X(j)\rangle_n\leq w_n=\sum\limits_{k=1}^n(a_k\mu_k)^2\quad \mathbb{P}\text{-a.s.}
    \end{equation*}
    Hence, we get by \cite[Theorem 4.3.15]{Duflo} that for all $\gamma>0$
    \begin{equation}
    \label{3.23}
        \frac{H^X_n(j)^2}{\langle H^X(j)\rangle_n}=o\big(\big(\log\langle H^X(j)\rangle_n\big)^{1+\gamma}\big)\quad \mathbb{P}\text{-a.s.}
    \end{equation}
    Since $\langle H^X(j)\rangle_n\leq w_n$ and by \eqref{3.23}, we obtain that
    \begin{equation*}
        H^X_n(j)^2=o\big(w_n\big(\log w_n\big)^{1+\gamma}\big)\quad \mathbb{P}\text{-a.s.}
    \end{equation*}
    In the diffusive regime, by Lemma \ref{LEM:wn-regimes} and \eqref{2.17}, we have
    \begin{equation*}
        H^X_n(j)^2=o\big(n^{1-2(a(\beta+1)-\beta)}\big(\log n\big)^{1+\gamma}\big)\quad \mathbb{P}\text{-a.s.}
    \end{equation*}
    By \eqref{2.11} and \eqref{2.12}, we observe that
    \begin{equation*}
        \bigg(\frac{H^X_n(j)}{na_n\mu_{n+1}}\bigg)^2=o\big(n^{-1}\big(\log n\big)^{1+\gamma}\big)\quad \mathbb{P}\text{-a.s.}
    \end{equation*}
    Hence
    \begin{equation*}
        \frac{H^X_n(j)}{na_n\mu_{n+1}}\to0\quad \text{as}\quad n\to\infty\quad \mathbb{P}\text{-a.s.}
    \end{equation*}
    By \eqref{2.11} and \eqref{2.12} again, we observe further
    \begin{equation}
    \label{3.29}
        \frac{1}{na_n\mu_{n+1}}\sum\limits_{k=1}^na_k\mu_k\to\frac{1}{ (1-a)(\beta+1)}\quad \text{as}\quad n\to\infty.
    \end{equation}
    Hence, we have
    \begin{equation*}
        \mu_{n+1}^{-1}\Lambda^X_n(j)\to\to\frac{1}{ \beta+1}\quad \text{as}\quad n\to\infty.
    \end{equation*}
    By \eqref{3.16} and \eqref{3.20}, we can then conclude that
    \begin{equation*}
        \frac{1}{n\mu_{n+1}}\Sigma_n\to\frac{1}{d (\beta+1)}I_d\quad \text{as}\quad n\to\infty\quad \mathbb{P}\text{-a.s.}
    \end{equation*}
    in the diffusive regime. In the critical regime, where $a=1-\frac{1}{2(\beta+1)}$,  we have from \eqref{2.19})
    \begin{equation*}
        H^X_n(j)^2=o\big(\log n\big(\log\log n\big)^{1+\gamma}\big)\quad \mathbb{P}\text{-a.s.}
    \end{equation*}
    Hence 
    \begin{equation*}
        \bigg(\frac{H^X_n(j)}{na_n\mu_{n+1}}\bigg)^2=o\big(n^{-1}\log n\big(\log\log n\big)^{1+\gamma}\big)\quad \mathbb{P}\text{-a.s.}
    \end{equation*}
    which implies that
    \begin{equation*}
        \frac{H^X_n(j)}{na_n\mu_{n+1}}\to0\quad \text{as}\quad n\to\infty\quad \mathbb{P}\text{-a.s.}
    \end{equation*}
    Similar to the convergence in \eqref{3.29}, in the critical regime, we observe
    \begin{equation*}
        \frac{1}{na_n\mu_{n+1}}\sum\limits_{k=1}^na_k\mu_k\to\frac{1}{2}\quad \mathbb{P}\text{-a.s.}
    \end{equation*}
    Hence, we conclude that
    \begin{equation*}
        \mu_{n+1}^{-1}\Lambda^X_n(j)\to\frac{1}{d(\beta+1)}\quad \;\text{and}\quad \;\frac{1}{n\mu_{n+1}}\Sigma_n\to\frac{1}{d(\beta+1)}I_d\quad \text{as}\quad n\to\infty\quad \mathbb{P}\text{-a.s.}
    \end{equation*}
    which proves \eqref{3.16new}. In the superdiffusive regime, we have
    \begin{equation*}
        H^X_n(j)^2=o\big(1\big)\quad \mathbb{P}\text{-a.s.}
    \end{equation*}
    and then
    \begin{equation*}
        \bigg(\frac{H^X_n(j)}{na_n\mu_{n+1}}\bigg)^2=o\big(n^{-2(1-a)(\beta+1)}\big)\quad \mathbb{P}\text{-a.s.}
    \end{equation*}
    which implies
    \begin{equation*}
        \frac{H^X_n(j)}{na_n\mu_{n+1}}\to0\quad \text{as}\quad n\to\infty\quad \mathbb{P}\text{-a.s.}
    \end{equation*}
    We can similarly show that
    \begin{equation*}
        \mu_{n+1}^{-1}\Lambda^X_n(j)\to\frac{1}{ \beta+1}\quad \text{as}\quad n\to\infty.
    \end{equation*}
    which then ensures that
    \begin{equation*}
        \frac{1}{n\mu_{n+1}}\Sigma_n\to\frac{1}{d (\beta+1)}I_d\quad \text{as}\quad n\to\infty\quad \mathbb{P}\text{-a.s.}
    \end{equation*}
    Consequently, the assertion is verified.
\end{proof}

The next result follows directly from the definition of $M_n$ and $N_n$
\begin{lemma}
\label{lem: quad-var}
    We have the following formulas for the predictable matrix-valued quadratic variations
    \begin{equation}
    \label{3.43}
        \langle M\rangle_n=(a_1\mu_1)^2\mathbb{E}\big[X_1X_1^T\big]+\sum\limits_{k=1}^{n-1}\frac{a(\beta+1)}{ka_{k+1}^{-2}}\mu_{k+1}\Sigma_k+\frac{1-a}{da_{k+1}^{-2}}\mu_{k+1}^2I_d-\bigg(\frac{\gamma_k-1}{a_{k+1}^{-1}}\bigg)^2Y_kY_k^T,
    \end{equation}
    and
    \begin{equation}
    \label{3.44}
        \langle N\rangle_n=\bigg(\frac{\beta}{\beta-a(\beta+1)}\bigg)^2\mathbb{E}\big[X_1X_1^T\big]+\sum\limits_{k=1}^{n-1}\frac{a(\beta+1)}{k\mu_{k+1}}\Sigma_k+\frac{1-a}{d}I_d-\bigg(\frac{\gamma_k-1}{\mu_{k+1}}\bigg)^2Y_kY_k^T.
    \end{equation}
    In particular, we have
    \begin{equation}
    \label{3.33}
        \Tr\langle M\rangle_n=w_n-\sum\limits_{k=1}^n(\gamma_k-1)^2a_{k+1}^2\norm{Y_k}^2,
    \end{equation}
    and
    \begin{equation}
    \label{3.34}
        \Tr\langle N\rangle_n=\bigg(\frac{\beta}{\beta-a(\beta+1)}\bigg)^2n-\sum\limits_{k=1}^{n-1}\bigg(\frac{a(\beta+1)}{k\mu_{k+1}}\bigg)^2\norm{Y_k}^2.
    \end{equation}
\end{lemma}


\begin{lemma}
\label{LEM:exp-epsilon}
    We have the following estimate for the matrix-valued conditional expectation.
    \begin{equation*}
        \mathbb{E}\big[\epsilon_{n+1}\epsilon_{n+1}^T|\mathcal{F}_n\big]=\frac{a(\beta+1)}{n}\mu_{n+1}\Sigma_n+\frac{1-a}{d}\mu_{n+1}^2I_d-(\gamma_n-1)^2Y_nY_n^T.
    \end{equation*}
    And as a consequence
    \begin{equation*}
        \mathbb{E}\big[\norm{\epsilon_{n+1}}^2|\mathcal{F}_n\big]=\mu_{n+1}^2-(\gamma_n-1)^2\norm{Y_n}^2.
    \end{equation*}
\end{lemma}
\begin{proof}
    Observe that
    \begin{equation*}
        \mathbb{E}\big[\epsilon_{n+1}\epsilon_{n+1}^T|\mathcal{F}_n\big]=\mathbb{E}\big[Y_{n+1}Y_{n+1}^T|\mathcal{F}_n\big]-\gamma_n^2Y_nY_n^T
    \end{equation*}
    with
    \begin{equation}\begin{aligned}
    \label{3.8}
        \mathbb{E}\big[Y_{n+1}Y_{n+1}^T|\mathcal{F}_n\big]&=Y_nY_n^T+2\mu_{n+1}Y_n\mathbb{E}\big[X_{n+1}^T|\mathcal{F}_n\big]+\mu_{n+1}^2\mathbb{E}\big[X_{n+1}X_{n+1}^T|\mathcal{F}_n\big]\\
        &=\bigg(1+\frac{2a(\beta+1)}{n}\bigg)Y_nY_n^T+\mu_{n+1}^2\mathbb{E}\big[X_{n+1}X_{n+1}^T|\mathcal{F}_n\big].
    \end{aligned}\end{equation}
    For all $k\geq1$, we know that $X_kX_k^T=\sum_{j=1}^d\mathbbm{1}_{\{X_k^j\neq0\}}e_je_j^T$. Then 
    \begin{equation*}\begin{aligned}
        &\mathbb{P}\big(X_{n+1}^j\neq0|\mathcal{F}_n\big)=\sum\limits_{k=1}^n\mathbb{P}\big(\beta_{n+1}=k\big)\cdot\mathbb{P}\big((A_nX_k)^j\neq0|\mathcal{F}_n\big)\\
        &\;=\sum\limits_{k=1}^n\mathbbm{1}_{\{X_k^j\neq0\}}\mathbb{P}\big(A_n=\pm I_d\big)\cdot\frac{(\beta+1)\mu_k}{n\mu_{n+1}}+\sum\limits_{k=1}^n\big(1-\mathbbm{1}_{\{X_k^j\neq0\}}\big)\mathbb{P}\big(A_n=\pm J_d\big)\cdot\frac{(\beta+1)\mu_k}{n\mu_{n+1}}.
    \end{aligned}\end{equation*}
    Hence
    \begin{equation}\begin{aligned}
    \label{3.10}
        \mathbb{P}\big(X_{n+1}^j\neq0|\mathcal{F}_n\big)&=\frac{\beta+1}{n\mu_{n+1}}\cdot\bigg(\mathbb{P}\big(A_n=+I_d\big)-\mathbb{P}\big(A_n=+J_d\big)\bigg)N^X_n(j)+2\mathbb{P}\big(A_n=+J_d\big)\\
        &=\frac{a(\beta+1)}{n\mu_{n+1}}N^X_n(j)+\frac{1-a}{d}.
    \end{aligned}\end{equation}
    Therefore
    \begin{equation}
    \label{3.11}
        \mathbb{E}\big[X_{n+1}X_{n+1}^T|\mathcal{F}_n\big]=\sum\limits_{j=1}^d\mathbb{P}\big(X_{n+1}^j\neq0|\mathcal{F}_n\big)e_je_j^T=\frac{a(\beta+1)}{n\mu_{n+1}}\Sigma_n+\frac{1-a}{d}I_d.
    \end{equation}
    And from \eqref{3.8} and \eqref{3.11} we can conclude that
    \begin{equation}\begin{aligned}
    \label{3.12}
        &\mathbb{E}\big[\epsilon_{n+1}\epsilon_{n+1}^T|\mathcal{F}_n\big]=\mathbb{E}\big[Y_{n+1}Y_{n+1}^T|\mathcal{F}_n\big]-\gamma_n^2Y_nY_n^T\\
        &\quad =\bigg(1+\frac{2a(\beta+1)}{n}\bigg)Y_nY_n^T+\frac{a(\beta+1)}{n}\mu_{n+1}\Sigma_n+\frac{1-a}{d}\mu_{n+1}^2I_d-\gamma_n^2Y_nY_n^T\\
        &\quad =\frac{a(\beta+1)}{n}\mu_{n+1}\Sigma_n+\frac{1-a}{d}\mu_{n+1}^2I_d-(\gamma_n-1)^2Y_nY_n^T.
    \end{aligned}\end{equation}
    On the other hand
    \begin{equation}
    \label{3.13}
        \Tr(\Sigma_n)=\frac{n\mu_{n+1}}{\beta+1}.
    \end{equation}
    Taking traces in \eqref{3.12} and by \eqref{3.13}, we have
    \begin{equation*}
        \mathbb{E}\big[\norm{\epsilon_{n+1}}^2|\mathcal{F}_n\big]=\mu_{n+1}^2-(\gamma_n-1)^2\norm{Y_n}^2
    \end{equation*}
    which ensures that the assertion is verified.
\end{proof}

\subsection{Scaling limits of the random walk and the barycenter}
\subsubsection{The diffusive regime}

\begin{lemma}
\label{lem:H1-D}
    For each $n\in\mathbb{N}$ and test vector $u\in\mathbb{R}^d$, let
    \begin{equation}
    \label{4.1}
        V_n=\frac{1}{\sqrt{n}}
        \begin{pmatrix}
            1&0\\
            0&\frac{a(\beta+1)}{\beta-a(\beta+1)}(a_n\mu_{n})^{-1}
        \end{pmatrix}\quad \;\text{and}\quad \;v=
        \begin{pmatrix}
            1\\-1
        \end{pmatrix}.
    \end{equation}
    Then 
    \begin{equation}
    \label{4.2}
        v^TV_n\mathcal{L}_n(u)=\frac{1}{\sqrt{n}}S_n(u).
    \end{equation}
    And for all $t\geq0$, we have
    \begin{equation}
    \label{4.3}
        V_n\langle\mathcal{L}(u)\rangle_{\lfloor nt\rfloor} V^T_n\to \frac{u^Tu}{d}V_t\quad \text{as}\quad n\to\infty\quad \mathbb{P}\text{-a.s.}
    \end{equation}
    where
    \begin{equation}
    \label{4.4}
        V_t=\frac{1}{(\beta-a(\beta+1))^2}
        \begin{pmatrix}
            \beta^2t&\frac{a\beta}{1-a} t^{1+\beta-a(\beta+1)}\\\frac{a\beta}{1-a} t^{1+\beta-a(\beta+1)}&\frac{a^2(\beta+1)^2}{1-2a(\beta+1)+2\beta} t^{1+2\beta-2a(\beta+1)}
        \end{pmatrix}.
    \end{equation}
\end{lemma}
\begin{proof}
    From Lemma \ref{LEM:exp-epsilon} and the fact that $\langle M(u)\rangle_n=u^T\langle M\rangle_nu$, we see that
    \begin{equation*}\begin{aligned}
        &\langle M(u)\rangle_{\lfloor nt\rfloor}=a_1^2\mu_1^2u^T\mathbb{E}\big[X_1X_1^T\big]u\\
        &\quad \quad +\sum\limits_{k=1}^{{\lfloor nt\rfloor}-1}\frac{a(\beta+1)}{k} a_{k+1}^2\mu_{k+1}u^T\Sigma_ku+\frac{1-a}{d} a_{k+1}^2\mu_{k+1}^2u^Tu-(\gamma_k-1)^2a_{k+1}^2u^TY_kY_k^Tu
    \end{aligned}\end{equation*}
    and
    \begin{equation*}\begin{aligned}
        &\langle N(u)\rangle_{\lfloor nt\rfloor}=\bigg(\frac{\beta}{\beta-a(\beta+1)}\bigg)^2u^T\mathbb{E}\big[X_1X_1^T\big]u\\
        &\quad \quad +\bigg(\frac{\beta}{\beta-a(\beta+1)}\bigg)^2\sum\limits_{k=1}^{{\lfloor nt\rfloor}-1}\frac{a(\beta+1)}{k\mu_{k+1}}u^T\Sigma_ku+\frac{1-a}{d}u^Tu-\bigg(\frac{\gamma_k-1}{\mu_{k+1}}\bigg)^2u^TY_kY^T_ku.
    \end{aligned}\end{equation*}
    Using a similar token and Lemma \ref{LEM:wn-regimes}, we can work out the off-diagonal entries in $\langle \mathcal{L}(u)\rangle_{\lfloor nt\rfloor}$, and we obtain that
    \begin{equation*}\begin{aligned}
        &\lim\limits_{n\to\infty}V_n\langle\mathcal{L}(u)\rangle_{\lfloor nt\rfloor} V^T_n\\
        &\quad \;=\lim\limits_{n\to\infty}\frac{u^Tu}{nd(\beta-a(\beta+1))^2}
        \begin{pmatrix}
            \beta^2\lfloor nt\rfloor & \frac{a(\beta+1)\beta}{a_n\mu_{n}}\sum_{k=0}^{\lfloor nt\rfloor-1} a_{k+1}\mu_{k+1}\\
            \frac{a(\beta+1)\beta}{a_n\mu_{n}}\sum_{k=0}^{\lfloor nt\rfloor-1} a_{k+1}\mu_{k+1} & \bigg(\frac{a(\beta+1)}{a_n\mu_n}\bigg)^2\sum_{k=0}^{\lfloor nt\rfloor-1}(a_{k+1}\mu_{k+1})^2 
        \end{pmatrix}\\
        &\quad \;=\frac{u^Tu}{d(\beta-a(\beta+1))^2}
        \begin{pmatrix}
            \beta^2t & \frac{a\beta}{1-a} t^{1-(a(\beta+1)-\beta)}\\
           \frac{a\beta}{1-a} t^{1-(a(\beta+1)-\beta)} & \frac{a^2(\beta+1)^2}{1-2(a(\beta+1)-\beta)} t^{1-2(a(\beta+1)-\beta)} 
        \end{pmatrix}=\frac{u^Tu}{d}V_t\quad \mathbb{P}\text{-a.s.}
    \end{aligned}\end{equation*}
    where the last equality is due to \eqref{2.11} and \eqref{2.12}. Thus, it implies that
    \begin{equation*}
        \frac{1}{na_n\mu_n}\sum\limits_{k=1}^na_k\mu_k\to\frac{1}{1-(a(\beta+1)-\beta)}\quad \;\text{and}\quad \;\frac{1}{n(a_n\mu_n)^2}\sum\limits_{k=1}^n(a_k\mu_k)^2\to\frac{1}{1-2(a(\beta+1)-\beta)}
    \end{equation*}
    as $n\to\infty$. Hence, equation \eqref{4.3} holds and the assertion is then verified.
\end{proof}

\begin{lemma}
\label{lem:H2-D}
    The MARW satisfies the Lindeberg condition in the diffusive regime. That is, for all $t\geq0$ and all $\epsilon>0$, 
    \begin{equation*}
        \sum\limits_{k=1}^{\lfloor nt\rfloor}\mathbb{E}\big[\norm{V_n\Delta\mathcal{L}_k(u)}^2\mathbbm{1}_{\{\norm{V_n\mathcal{L}_k(u)}^2>\epsilon\}}|\mathcal{F}_{k-1}\big]\to0\quad \text{as}\quad n\to\infty\quad \mathbb{P}\text{-a.s.}
    \end{equation*}
\end{lemma}
\begin{proof}
    On the one hand, it is easy to compute from \eqref{3.3} and \eqref{4.1} that, for all $1\leq k\leq n$,
    \begin{equation*}
        V_n\Delta\mathcal{L}_k(u)=\frac{1}{\sqrt{n}(\beta-a(\beta+1))\mu_n}
        \begin{pmatrix}
            \beta\frac{\mu_n}{\mu_k}\\a\frac{a_k}{a_n}
        \end{pmatrix}\epsilon_k(u)
    \end{equation*}
    which implies
    \begin{equation*}
        \norm{V_n\Delta\mathcal{L}_k(u)}^2=\frac{1}{n(\beta-a(\beta+1))^2}\bigg(\frac{\beta^2}{\mu_k^2}+\frac{a^2a_k^2}{(a_n\mu_n)^2}\bigg)\epsilon_k(u)^2.
    \end{equation*}
    Hence
    \begin{equation}
    \label{4.12}
        \norm{V_n\Delta\mathcal{L}_k(u)}^4\leq\frac{2}{n^2(\beta-a(\beta+1))^4}\bigg(\frac{\beta^4}{\mu_k^4}+\frac{a^4a_k^4}{(a_n\mu_n)^4}\bigg)\epsilon_k(u)^4.
    \end{equation}
    On the other hand, from \eqref{2.11} we observe that
    \begin{equation}
    \label{4.13}
        \frac{1}{na_n^2}\sum\limits_{k=1}^na_k^2\leq C_1(a,\beta)^{-1}\quad \;\text{and}\quad \;\frac{1}{na_n^4}\sum\limits_{k=1}^na_k^4\leq C_2(a,\beta)^{-1}\quad \text{for all}\quad n\in\mathbb{N}
    \end{equation}
    and where $C_1(a,\beta),C_2(a,\beta)>0$ are constants depending only on $a$ and $\beta$. Moreover, we get that
    \begin{equation}
    \label{4.14}
        \sup\limits_{1\leq k\leq n}\abs{\epsilon_k(u)}\leq\sup\limits_{1\leq k\leq n}\norm{\epsilon_k}\norm{u}\leq\sup\limits_{1\leq k\leq n}(\beta+2)\mu_k\norm{u}\leq(\beta+2)\mu_n\norm{u}.
    \end{equation}
    Hence, we deduce from \eqref{4.13} and \eqref{4.14} 
    \begin{equation}
    \label{4.15}
        \sum\limits_{k=1}^n\norm{V_n\Delta\mathcal{L}_k(u)}^4\leq\frac{2}{n^2(\beta-a(\beta+1))^4}\bigg(\big(\beta(\beta+2)\big)^4\norm{u}^4+\frac{\big(a(\beta+2)\big)^4\norm{u}^4}{C_2(a,\beta)}\bigg)\to0
    \end{equation}
    as $n\to\infty$ $\mathbb{P}$-a.s. This implies that
    \begin{equation*}
        \sum\limits_{k=1}^n\mathbb{E}\big[\norm{V_n\Delta\mathcal{L}_k(u)}^4|\mathcal{F}_{k-1}\big]\to0\quad \text{as}\quad n\to\infty\quad \mathbb{P}\text{-a.s.}
    \end{equation*}
    Therefore, for all $\epsilon>0$, we obtain
    \begin{equation*}
        \sum\limits_{k=1}^{n}\mathbb{E}\big[\norm{V_n\Delta\mathcal{L}_k(u)}^2\mathbbm{1}_{\{\norm{V_n\mathcal{L}_k(u)}^2>\epsilon\}}|\mathcal{F}_{k-1}\big]\leq\frac{1}{\epsilon^2}\sum\limits_{k=1}^n\mathbb{E}\big[\norm{V_n\Delta\mathcal{L}_k(u)}^4|\mathcal{F}_{k-1}\big]\to0
    \end{equation*}
    as $n\to\infty$ $\mathbb{P}$-a.s. This yields finally
    \begin{equation*}
        \sum\limits_{k=1}^{\lfloor nt\rfloor}\mathbb{E}\big[\norm{V_n\Delta\mathcal{L}_k(u)}^2\mathbbm{1}_{\{\norm{V_n\mathcal{L}_k(u)}^2>\epsilon\}}|\mathcal{F}_{k-1}\big]\leq\frac{1}{\epsilon^2}\sum\limits_{k=1}^{\lfloor nt\rfloor}\mathbb{E}\bigg[\norm{(V_nV_{\lfloor nt\rfloor}^{-1})V_{\lfloor nt\rfloor}\Delta\mathcal{L}_k(u)}^4|\mathcal{F}_{k-1}\bigg]\to0
    \end{equation*}
    as $n\to\infty$ $\mathbb{P}$-a.s. since $V_nV_{\lfloor nt\rfloor}^{-1}$ converges as $n\to\infty$. 
\end{proof}

\begin{lemma}
\label{lem:H3-D}
    The deterministic matrix $V_t$ defined in \eqref{4.4} can be rewritten as
    \begin{equation*}
        V_t=t^{\alpha_1}K_1+t^{\alpha_2}K_2+\cdots+ t^{\alpha_q}K_q
    \end{equation*}
    with $q\in\mathbb{N}$, $\alpha_j>0$ and each $K_j$ is a symmetric matrix for all $1\leq j\leq 1$.
\end{lemma}
\begin{proof}
    A direct computation analoguous to the one in \cite{Laulin2} shows that $V_t=t^{\alpha_1}K_1+t^{\alpha_2}K_2+t^{\alpha_3}K_3$, where
    \begin{equation*}
        \alpha_1=1,\quad \;\alpha_2=1-a(\beta+1)>0,\quad \;\alpha_3=1-2a(\beta+1)>0
    \end{equation*}
    since $a<1-\frac{1}{2(\beta+1)}$ is in the diffusive regime. Moreover
    \begin{equation*}\begin{aligned}
        K_1=&\frac{\beta^2}{(a(\beta+1)-\beta)^2}
        \begin{pmatrix}
            1&0\\0&0
        \end{pmatrix},\quad \;
        K_2=\frac{a\beta}{(1-a)(a(\beta+1)-\beta)^2}
        \begin{pmatrix}
            0&1\\1&0
        \end{pmatrix},\\
        &\quad \;K_3=\frac{a^2(\beta+1)^2}{(1-2a(\beta+1)+2\beta)(a(\beta+1)-\beta)^2}
        \begin{pmatrix}
            0&0\\0&1
        \end{pmatrix}.
    \end{aligned}\end{equation*}
\end{proof}

\begin{lemma}
\label{lem:H4-D}
   Given the matrix-valued process $(V_n)_{n\in\mathbb{N}}$ define in \eqref{4.1}, we have
    \begin{equation*}
        \sum\limits_{n=1}^\infty\frac{1}{\big(\log(\det V_n^{-1})^2\big)^2}\mathbb{E}\big[\norm{V_n\Delta\mathcal{L}_n(u)}^4|\mathcal{F}_{n-1}\big]<\infty\quad \mathbb{P}\text{-a.s.}
    \end{equation*}
\end{lemma}
\begin{proof}
    From \eqref{4.1}, it is immediate that 
    \begin{equation}
    \label{4.23old}
        \det V^{-1}_n=\frac{\beta-a(\beta+1)}{a(\beta+1)}na_n\mu_n.
    \end{equation}
    By \eqref{2.11} and \eqref{2.12}, we obtain
    \begin{equation}
    \label{4.24}
        \frac{\log(\det V^{-1}_n)^2}{\log n}\to2(1-a)(\beta+1)\quad \text{as}\quad n\to\infty\quad \mathbb{P}\text{-a.s.}
    \end{equation}
    Hence there exists a constant $C(a,\beta)>0$ depending only on $a$ and $\beta$ such that
    \begin{equation}
    \label{4.25}
        \sum\limits_{n=1}^\infty\frac{1}{\big(\log(\det V_n^{-1})^2\big)^2}\mathbb{E}\big[\norm{V_n\Delta\mathcal{L}_n(u)}^4|\mathcal{F}_{n-1}\big]\leq C(a,\beta)\sum\limits_{n=1}^\infty\frac{1}{(\log n)^2}\mathbb{E}\big[\norm{V_n\Delta\mathcal{L}_n(u)}^4|\mathcal{F}_{n-1}\big].
    \end{equation}
    Hereafter, equations \eqref{4.12}, \eqref{4.14}, \eqref{4.15} together imply that
    \begin{equation}
    \label{eq:log-Vn}
        \sum\limits_{n=1}^\infty\frac{1}{(\log n)^2}\norm{V_n\Delta\mathcal{L}_n(u)}^4\leq C'(a,\beta)\sum\limits_{n=1}^\infty\frac{1}{(n\log n)^2}<\infty\quad \mathbb{P}\text{-a.s.}
    \end{equation}
    for some other constant $C'(a,\beta)>0$ depending only on $a$ and $\beta$. Consequently, equation \eqref{eq:log-Vn} together \eqref{4.25} ensures that the assertion is verified.
\end{proof}
\subsubsection{The critical regime}
\begin{lemma}
\label{lem:H1-C}
    For each $n\in\mathbb{N}$ and test vector $u\in\mathbb{R}^d$, let
    \begin{equation}
    \label{4.49}
        W_n=\frac{1}{\sqrt{n\log n}}
        \begin{pmatrix}
            1&0\\
            0&\frac{2\beta+1}{a_{n} \mu_{n}}
        \end{pmatrix}\quad \;\text{and}\quad \;w=
        \begin{pmatrix}
            1\\-1
        \end{pmatrix}.
    \end{equation}
    Then for all $t\geq0$, we have
    \begin{equation}
    \label{4.50}
        w^TW_n\mathcal{L}_n(u)=\frac{1}{\sqrt{n\log n}}S_n(u)
    \end{equation}
    and
    \begin{equation}
    \label{4.51}
        W_n\langle\mathcal{L}(u)\rangle_{n} W_n^T\to \frac{u^Tu}{d}W\quad \text{as}\quad n\to\infty\quad \mathbb{P}\text{-a.s.}\quad \;\text{where}\quad \; W_t=(2\beta+1)^2
        \begin{pmatrix}
            0&0\\0&1
        \end{pmatrix}.
    \end{equation}
\end{lemma}
\begin{proof}
    It is clear that \eqref{4.50} follows from \eqref{2.16}. Using a similar token than for the proof Lemma \ref{lem:H1-D}, we have
    \begin{equation*}\begin{aligned}
        &\lim\limits_{n\to\infty}W_n\langle\mathcal{L}(u)\rangle_{n} W_n^T\\
        &\quad \quad =\lim\limits_{n\to\infty}\frac{4u^Tu}{(n\log n)d}
        \begin{pmatrix}
            \beta^2n & \frac{\beta(\beta+\frac{1}{2})}{a_{n}\mu_{n }}\sum_{k=0}^{n-1} a_{k+1}\mu_{k+1}\\
            \frac{\beta(\beta+\frac{1}{2})}{a_{n}\mu_{n }}\sum_{k=0}^{n-1} a_{k+1}\mu_{k+1} & \bigg(\frac{\beta+\frac{1}{2}}{a_{n}\mu_{n }}\bigg)^2\sum_{k=0}^{n-1}(a_{k+1}\mu_{k+1})^2 
        \end{pmatrix}\\
        &\quad \quad =\frac{4u^Tu}{d}
        \begin{pmatrix}
            0 & 0\\
           0 & \big(\beta+\frac{1}{2}\big)^2 
        \end{pmatrix}=\frac{u^Tu}{d}W\quad \mathbb{P}\text{-a.s.}
    \end{aligned}\end{equation*}
    and the proof is complete.
\end{proof}
\begin{lemma}
\label{lem:H2-C}
    The MARW satisfies the Lindeberg condition in the critical regime. That is, for all $t\geq0$ and all $\epsilon>0$, given the $(W_n)_{n\in\mathbb{N}}$ defined in \eqref{4.1}, it satisfies
    \begin{equation*}
        \sum\limits_{k=1}^{n}\mathbb{E}\big[\norm{W_n\Delta\mathcal{L}_k(u)}^2\mathbbm{1}_{\{\norm{W_n\mathcal{L}_k(u)}^2>\epsilon\}}|\mathcal{F}_{k-1}\big]\to0\quad \text{as}\quad n\to\infty\quad \mathbb{P}\text{-a.s.}
    \end{equation*}
\end{lemma}
\begin{proof}
    We state that equations \eqref{4.12} and \eqref{4.13} remain true with $V_n$ replaced by $W_n$. More precisely, they can be rewritten as
    \begin{equation}
    \label{4.53new}
        \norm{W_n\Delta\mathcal{L}_k(u)}^4\leq\frac{32}{(n\log n)^2} \bigg(\frac{\beta^4}{\mu_k^4}+\frac{a^4a_k^4}{(a_n\mu_n)^4}\bigg)\epsilon_k(u)^4
    \end{equation}
    and
    \begin{equation*}
        \frac{1}{na_n^4}\sum\limits_{k=1}^{n} a_k^4\leq C(a,\beta)^{-1}\quad \text{for all}\quad n\in\mathbb{N}
    \end{equation*}
    where $C(a,\beta)>0$ is a constant depending only on $t$, $a$, and $\beta$. Since \eqref{4.14} is not affected by switching regimes, we have that
    \begin{equation}
    \label{4.55}
        \sum\limits_{k=1}^{n}\norm{W_n\Delta\mathcal{L}_k(u)}^4\leq\frac{32}{(n\log n)^2}\bigg(\big(\beta(\beta+2)\big)^4\norm{u}^4+\frac{\big(a(\beta+2)\big)^4\norm{u}^4}{C(t,a,\beta)}\bigg)\to0
    \end{equation}
    as $n\to\infty$ $\mathbb{P}$-a.s. This implies
    \begin{equation*}
        \sum\limits_{k=1}^{n}\mathbb{E}\big[\norm{W_n\Delta\mathcal{L}_k(u)}^4|\mathcal{F}_{k-1}\big]\to0\quad \text{as}\quad n\to\infty\quad \mathbb{P}\text{-a.s.}
    \end{equation*}
    Therefore, for all $\epsilon>0$, we obtain
    \begin{equation*}
        \sum\limits_{k=1}^{n}\mathbb{E}\big[\norm{W_n\Delta\mathcal{L}_k(u)}^2\mathbbm{1}_{\{\norm{W_n\mathcal{L}_k(u)}^2>\epsilon\}}|\mathcal{F}_{k-1}\big]\leq\frac{1}{\epsilon^2}\sum\limits_{k=1}^{n}\mathbb{E}\big[\norm{W_n\Delta\mathcal{L}_k(u)}^4|\mathcal{F}_{k-1}\big]\to0
    \end{equation*}
    as $n\to\infty$ $\mathbb{P}$-a.s. and the assertion is verified.
\end{proof}

\begin{lemma}
\label{lem:H4-C}
   Given the matrix-valued sequence $(W_n)_{n\in\mathbb{N}}$ define in \eqref{4.49}, we have
    \begin{equation*}
        \sum\limits_{n=1}^\infty\frac{1}{\big(\log(\det W_n^{-1})^2\big)^2}\mathbb{E}\big[\norm{W_n\Delta\mathcal{L}_n(u)}^4|\mathcal{F}_{n-1}\big]<\infty\quad \mathbb{P}\text{-a.s.}
    \end{equation*}
\end{lemma}
\begin{proof}
    From \eqref{4.49}, it is immediate that 
    \begin{equation}
    \label{4.59}
        \det W_n^{-1}=\frac{1}{2\beta+1}\sqrt{n\log n}\cdot a_{n}\mu_{n}.
    \end{equation}
    Then, we obtain by \eqref{2.11} and \eqref{2.12} that
    \begin{equation}
    \label{4.60}
        \frac{\log(\det W_n^{-1})^2}{\log\log n}\to1\quad \text{as}\quad n\to\infty\quad \mathbb{P}\text{-a.s.}
    \end{equation}
    Hence, there exists a constant $C(a,\beta)>0$ depending only on $a$ and $\beta$ such that
    \begin{equation}
    \label{4.61}
            \sum\limits_{n=1}^\infty\frac{1}{\big(\log(\det W_n^{-1})^2\big)^2}\mathbb{E}\big[\norm{W_n\Delta\mathcal{L}_n(u)}^4|\mathcal{F}_{n-1}\big]\leq \sum\limits_{n=1}^\infty\frac{C(a,\beta)}{(\log\log n)^2}\mathbb{E}\big[\norm{W_n\Delta\mathcal{L}_n(u)}^4|\mathcal{F}_{n-1}\big].
    \end{equation}
    Hereafter, \eqref{4.53new} together with \eqref{4.55} imply that
    \begin{equation*}
        \sum\limits_{n=1}^\infty\frac{1}{(\log\log n)^2}\norm{W_n\Delta\mathcal{L}_n(u)}^4\leq C'(a,\beta)\sum\limits_{n=1}^\infty\frac{1}{(n\log n\log\log n)^2}<\infty\quad \mathbb{P}\text{-a.s.}
    \end{equation*}
    for some other constant $C'(a,\beta)>0$ depending only on $a$ and $\beta$. Finally, using the above equation together with \eqref{4.61} completes the proof.
\end{proof}

\begin{lemma}
    Fix the test vector $u\in\mathbb{R}^d$. The growth rate of the compensator of the partial sum of $(N_n(u)^2)_{n\in\mathbb{N}}$ is less than cubic growth, in the sense that
    \begin{equation*}
        \frac{1}{n^3}\sum\limits_{k=1}^{n-1}\mathbb{E}\big[N_{k+1}(u)^2|\mathcal{F}_n\big]\to0\quad \text{as}\quad n\to\infty\quad \mathbb{P}\text{-a.s.}
    \end{equation*}
\end{lemma}
\begin{proof}
    The law of iterated expectations and \eqref{3.44} yields
    \begin{equation*}
        \frac{1}{n}\mathbb{E}\big[\mathbb{E}\big[N_{n+1}(u)^2|\mathcal{F}_n\big]\big]=\frac{1}{n}\mathbb{E}\big[\langle N(u)\rangle_n\big]\to\bigg(\frac{\beta}{\beta-a(\beta+1)}\bigg)^2u^Tu\quad \text{as}\quad n\to\infty\quad \mathbb{P}\text{-a.s.}
    \end{equation*}
    The strong law of large numbers then yields
    \begin{equation*}
        \frac{1}{n}\sum_{k=1}^{n-1}\frac{1}{k}\mathbb{E}\big[N_{k+1}(u)^2|\mathcal{F}_k\big]\to\bigg(\frac{\beta}{\beta-a(\beta+1)}\bigg)^2u^Tu\quad \text{as}\quad n\to\infty\quad \mathbb{P}\text{-a.s.}
    \end{equation*}
    Hence
    \begin{equation*}
        \frac{1}{n^3}\sum\limits_{k=1}^{n-1}\mathbb{E}\big[N_{k+1}(u)^2|\mathcal{F}_n\big]\leq\frac{1}{n^2}\sum_{k=1}^{n-1}\frac{1}{k}\mathbb{E}\big[N_{k+1}(u)^2|\mathcal{F}_k\big]\to0\quad \text{as}\quad n\to\infty\quad \mathbb{P}\text{-a.s.}
    \end{equation*}
\end{proof}
\subsubsection{The barycenter process}
For the following Toeplitz Lemmas, see \cite{Duflo} and \cite{Li}.

\begin{lemma}\cite[Theorem 1.1 Part I]{Li} 
\label{lem:toep1}
    Let $(a_{n,k})_{1\leq k\leq k_n,\,n\in\mathbb{N}}$ be a double array of real numbers such that for all $k\geq1$, we have $a_{n,k}\to0$ as $n\to\infty$ and $\sup_{n\in\mathbb{N}}\sum_{k=1}^{k_n}\abs{a_{n,k}}<\infty$. Let $(x_n)_{n\in\mathbb{N}}$ be a real sequence. If $x_n\to0$ as $n\to\infty$, then $\sum_{k=1}^{k_n} a_{n,k}x_k\to0$ as $n\to\infty$.
\end{lemma}

\begin{lemma}\cite[Theorem 1.1 Part II]{Li} 
\label{lem:toep2}
    Let $(a_{n,k})_{1\leq k\leq k_n,\,n\in\mathbb{N}}$ be a double array of real numbers such that for all $k\geq1$, we have $a_{n,k}\to0$ as $n\to\infty$ and $\sup_{n\in\mathbb{N}}\sum_{k=1}^{k_n}\abs{a_{n,k}}<\infty$. Let $(x_n)_{n\in\mathbb{N}}$ be a real sequence. If $x_n\to x$ as $n\to\infty$ with $x\in\mathbb{R}$ and $\sum_{k=1}^{k_n} a_{n,k}=1$, then $\sum_{k=1}^{k_n} a_{n,k}x_k\to x$ as $n\to\infty$.
\end{lemma}

\subsection{Quadratic rate estimates}
 Our first result is about the convergence rate of the process $(Y_n)_{n\in\mathbb{N}}$ defined in \eqref{2.8}.
\begin{lemma}
\label{lem: 6.0.1}
    For all $p\in(0,1)$, then we have, as $n\to\infty$,
    \begin{equation*}
        \mathbb{E}[Y_nY_n^T]\sim\frac{n^{2a(\beta+1)}}{\Gamma(1+2a(\beta+1))}\cdot\frac{1}{d}Id+\frac{n^{1+2\beta}}{\Gamma(\beta+1)^2(1+2\beta-2a(\beta+1))(\beta+1)}\cdot\frac{1}{d}Id.
    \end{equation*}
\end{lemma}
\begin{proof}
    From \eqref{3.8} and \eqref{3.11}, we see
    \begin{equation*}
        \mathbb{E}\big[Y_{n+1}Y_{n+1}^T|\mathcal{F}_n\big]=\bigg(1+\frac{2a(\beta+1)}{n}\bigg)Y_nY_n^T+\mu_{n+1}^2\bigg(\frac{a(\beta+1)}{n\mu_{n+1}}\Sigma_n+\frac{1-a}{d}Id\bigg).
    \end{equation*}
    Then, remember that
    \begin{equation*}
        \mathbb{E}\big[\Sigma_n\big]=\sum\limits_{j=1}^d\mathbb{E}\big[N^X_n(j)\big]e_je_j^T=\sum\limits_{j=1}^d\sum\limits_{k=1}^n\mathbb{P}\big(X^j_k\neq0\big)\mu_k\cdot e_je_j^T.
    \end{equation*}
    Lemma \ref{LEM:wn-regimes} yields $\mathbb{E}[(n\mu_{n+1})^{-1}\Sigma_n]\sim(\beta+1)^{-1}\cdot\tfrac{1}{d}Id$. Hence,
    \begin{equation*}
        \mathbb{E}\big[Y_{n+1}Y_{n+1}^T\big]\sim\bigg(1+\frac{2a(\beta+1)}{n}\bigg)\mathbb{E}\big[Y_nY_n^T\big]+\frac{\mu_{n+1}^2}{\beta+1}\cdot\frac{1}{d}Id.
    \end{equation*}
    A recursive argument then gives
    \begin{equation*}\begin{aligned}
        \mathbb{E}\big[Y_nY_n^T\big]&\sim\frac{\Gamma(n+2a(\beta+1))}{\Gamma(n)\Gamma(1+2a(\beta+1))}\mathbb{E}\big[Y_1Y_1^T\big]+\sum\limits_{j=1}^{n-1}\frac{\mu_j^2}{\beta+1}\cdot\frac{\prod_{k=1}^{n-1}(1+k^{-1}2a(\beta+1))}{\prod_{k=1}^{j-1}(1+k^{-1}2a(\beta+1))}\cdot\frac{1}{d}Id\\
        &\sim\frac{\Gamma(n+2a(\beta+1))}{\Gamma(n)\Gamma(1+2a(\beta+1))}\cdot\frac{1}{d}Id+\sum\limits_{j=1}^{n-1}\frac{\mu_j^2}{\beta+1}\cdot\frac{\Gamma(n+2a(\beta+1))\Gamma(j)}{\Gamma(j+2a(\beta+1))\Gamma(n)}\cdot\frac{1}{d}Id.
    \end{aligned}\end{equation*}
    Employing the asymptotics in \eqref{2.6} and \eqref{2.12}, the assertion follows.
\end{proof}
The process $Y_n=\sum_{k=1}^n\mu_kX_k$ differs from $S_n$ by a multiplicative factor at each step. When there is no amnesia, the asymptotics of these two processes coincide. However, when $\beta\geq0$, we have to treat the general case in another way.
\begin{lemma}
\label{lem: 6.0.2}
    For all $p\in(0,1)$ and test vector $u\in\mathbb{R}^d$, we have, as $n\to\infty$,
    \begin{equation*}
        \mathbb{E}\big[\langle M(u)\rangle_n\big]\sim w_nu^Tu-(C_1n^{-1}+C_2n^{-2(a(\beta+1)-\beta)})u^Tu,
    \end{equation*}
    and
    \begin{equation*}
        \mathbb{E}\big[\langle N(u)\rangle_n\big]\sim\bigg(\frac{\beta}{\beta-a(\beta+1)}\bigg)^2nu^Tu-(C_1n^{1-2(1-a)(\beta+1)}+C_2)u^Tu.
    \end{equation*}
\end{lemma}
\begin{proof}
    By Lemma \ref{lem: quad-var}
    \begin{equation*}
        \mathbb{E}\big[\langle M(u)\rangle_n\big]=\mathbb{E}\big[\Tr\langle M\rangle_n\big]u^Tu=w_nu^Tu-\sum\limits_{k=1}^n(\gamma_k-1)^2a_{k+1}^2u^T\mathbb{E}\big[Y_kY_k^T\big]u.
    \end{equation*}
    By Lemma \ref{lem: 6.0.1} and a finite summation,
    \begin{equation*}\begin{aligned}
        \mathbb{E}\big[\langle M(u)\rangle_n\big]&\sim w_nu^Tu-\sum\limits_{k=1}^{n-1}\frac{a^2(\beta+1)^2}{k^2}(k+1)^{-2a(\beta+1)}(C_1k^{2a(\beta+1)}+C_2k^{1+2\beta})u^Tu\\
        &\sim w_nu^Tu-(C_1n^{-1}+C_2n^{-2(a(\beta+1)-\beta)})u^Tu.
    \end{aligned}\end{equation*}
    Similarly,
    \begin{equation*}
        \mathbb{E}\big[\langle N(u)\rangle_n\big]=\mathbb{E}\big[\Tr\langle N\rangle_n\big]u^Tu=\bigg(\frac{\beta}{\beta-a(\beta+1)}\bigg)^2nu^Tu-\sum\limits_{k=1}^{n-1}\frac{a^2(\beta+1)^2}{k^2}\mu_{k+1}^{-2}u^T\mathbb{E}\big[Y_kY_k^T\big]u.
    \end{equation*}
    Hence, using Lemma \ref{lem: 6.0.1} again, we observe
    \begin{equation*}\begin{aligned}
        \mathbb{E}\big[\langle N(u)\rangle_n\big]&\sim\bigg(\frac{\beta}{\beta-a(\beta+1)}\bigg)^2nu^Tu-\sum\limits_{k=1}^{n-1}\frac{a^2(\beta+1)^2}{k^2}(k+1)^{-2\beta}(C_1k^{2a(\beta+1)}+C_2k^{1+2\beta})u^Tu\\
        &\sim\bigg(\frac{\beta}{\beta-a(\beta+1)}\bigg)^2nu^Tu-(C_1n^{1-2(1-a)(\beta+1)}+C_2)u^Tu.
    \end{aligned}\end{equation*}
\end{proof}
\begin{lemma}
\label{lem: 6.0.3}
    For all $p\in(0,1)$ and test vector $u\in\mathbb{R}^d$, we have, as $n\to\infty$,
    \begin{equation*}\begin{aligned}
        \mathbb{E}\big[\langle M(u),N(u)\rangle_n\big]&\sim\frac{\beta}{\beta-a(\beta+1)}\cdot\frac{\Gamma(\beta+1)\Gamma(a(\beta+1)+1)}{(1-a)(\beta+1)}n^{(1-a)(\beta+1)}u^Tu\\
        &\quad \quad \quad -(C_1n^{-(1-a)(\beta+1)}+C_2n^{(1-a)(\beta+1)-1})u^Tu.
    \end{aligned}\end{equation*}
\end{lemma}
\begin{proof}
    By \eqref{3.3} and Lemma \ref{lem: quad-var}, for all test vector $u\in\mathbb{R}^d$
    \begin{equation*}
        \Delta\mathcal{L}_{n+1}(u)=\bigg(\frac{\beta\mu_{n+1}^{-1}}{\beta-a(\beta+1)}\bigg)^T\epsilon_{n+1}(u),
    \end{equation*}
    and therefore,
    \begin{equation*}
        \langle M(u),N(u)\rangle_n=\sum\limits_{k=1}^n\frac{\beta}{\beta-a(\beta+1)} a_k\mu_k^{-1}\mathbb{E}\bigg[\epsilon_k(u)\epsilon_k(u)^T|\mathcal{F}_{k-1}\bigg].
    \end{equation*}
    Taking the trace will give us
    \begin{equation*}
    \Tr\langle M,N\rangle_n=\frac{\beta}{\beta-a(\beta+1)}\sum\limits_{k=1}^na_k\mu_k-\frac{\beta}{\beta-a(\beta+1)}\sum\limits_{k=1}^na_k\mu_k^{-1}(\gamma_k-1)^2\norm{Y_k}^2.
    \end{equation*}
    Taking the expectation and using Lemma \ref{lem: 6.0.1} completes the proof.
\end{proof}

\subsection{Moderate deviations}
\begin{lemma}
    For all $p\in(0,1)$ and for all $j=1,\ldots,d$,
    \begin{equation}
        \label{8.1}
        \abs{\Delta M_n^j}\leq\big(a(\beta+1)+1\big)a_n\mu_n\quad \text{for all}\quad n\in\mathbb{N}.
    \end{equation}
\end{lemma}
\begin{proof}
    By \eqref{2.8} and \eqref{2.13}, 
    \begin{equation*}
        \Delta M_n^j=a_nY_n^j-a_{n-1}Y_{n-1}^j=a_n\mu_nX_n^j-(a_n-a_{n-1})\sum\limits_{k=1}^{n-1}\mu_kX_k^j.
    \end{equation*}
    Since $\norm{X_k}=1$ for eack $k\leq n$, then by \eqref{2.10},
    \begin{equation*}
        \abs{\Delta M_n^j}\leq a_n\mu_n+(n-1)(a_{n-1}-a_n)\mu_{n-1}\leq a_n\mu_n+a(\beta+1)a_n\mu_n.
    \end{equation*}
    And the assertion is verified.
\end{proof}
\begin{lemma}
    For all $p\in(0,1)$ and for all $j=1,\ldots,d$,
    \begin{equation*}\label{temp001}
        \abs{\Delta N_n^j}\leq2a(\beta+1)+\frac{\beta}{\beta-a(\beta+1)}\quad \text{for all}\quad n\in\mathbb{N}.
    \end{equation*}
\end{lemma}
\begin{proof}
    By \eqref{2.8} and \eqref{3.2},
    \begin{equation*}
        \Delta N^j_n=\frac{\beta\mu_{n+1}^{-1}}{\beta-a(\beta+1)}\epsilon^j_{n+1}=\frac{\beta\mu_{n+1}^{-1}}{\beta-a(\beta+1)}\cdot\big(\mu_{n+1}X^j_{n+1}+(1-\gamma_n)\sum\limits_{k=1}^nX_k^j\mu_k\big).
    \end{equation*}
    Taking absolute value on both sides, and the assertion is verified.
\end{proof}
\begin{lemma}
    For all $p\in(0,1)$ and for all $j=1,\ldots,d$,
    \begin{equation}
        \label{8.4}
        \abs{\frac{1}{\sqrt{w_n}}\Delta M^j_k}\leq\big(a(\beta+1)+1\big)\frac{a_n\mu_n}{\sqrt{w_n}}\quad \text{for each}\quad 1\leq k\leq n,
    \end{equation}
    and in the diffusive and critical regime,
    \begin{equation*}
        \abs{\frac{1}{w_n}\langle M^j\rangle_n-1}\leq
        \begin{cases}
            C\cdot n^{-1} & \text{when}\quad a<1-\frac{1}{2(\beta+1)}\\
            C\cdot (\log n)^{-1} & \text{when}\quad a=1-\frac{1}{2(\beta+1)}.
        \end{cases}
    \end{equation*}
\end{lemma}
\begin{proof}
    Dividing by $\sqrt{w_n}$ from both sides of \eqref{8.1}, we get \eqref{8.4}. Moreover, by \eqref{3.33},
    \begin{equation*}
        \abs{\langle M^j\rangle_n-w_n}\leq\sum\limits_{k=1}^n(\gamma_k-1)^2a_{k+1}^2\norm{Y_k}^2\leq C\sum\limits_{k=1}^n\frac{w_k}{k^2}.
    \end{equation*}
    Dividing both sides by $w_n$ and following \eqref{2.17}, \eqref{2.19}, the assertion is verified.
\end{proof}
\begin{lemma}
    For all $p\in(0,1)$ and for all $j=1,\ldots,d$,
    \begin{equation}
        \label{8.9}
        \abs{\frac{a_n\mu_n}{\sqrt{w_n}}\Delta N^j_k}\leq\big(2a(\beta+1)+\frac{\beta}{\beta-a(\beta+1)}\big)\frac{a_n\mu_n}{\sqrt{w_n}}\quad \text{for each}\quad 1\leq k\leq n,
    \end{equation}
    and in both the diffusive and critical regime,
    \begin{equation*}
        \abs{\frac{a_n^2\mu_n^2}{w_n}\langle N^j\rangle_n-1}\leq
        \begin{cases}
            C\cdot n^{-2(1-a)(\beta+1)} & \text{when}\quad a<1-\frac{1}{2(\beta+1)}\\
            C\cdot (n\log n)^{-1} & \text{when}\quad a=1-\frac{1}{2(\beta+1)}.
        \end{cases}
    \end{equation*}
\end{lemma}
\begin{proof}
    Dividing by $\sqrt{w_n}$ and multiplied by $a_n\mu_u$ from both sides of \eqref{temp001}, we get \eqref{8.9}. Then, by \eqref{3.34}, we make use of the estimates and the inequalities hold.
\end{proof}
Denote by $\Phi(\cdot)\coloneqq(2\pi)^{-1/2}\int_{-\infty}^\cdot e^{-t^2/2}\,dt$ the cumulative distribution of the standard normal random variable. The following lemmas are straightforward derivations from \cite[Theorem 1]{Fan2}, see also \cite{Grama}.
\begin{lemma}\label{lem: 7.0.5}
    There exists an absolute constant $\alpha^\prime(p,\beta)>0$ depending only on $p,\beta$ such that for all $j=1,\ldots,d$ and all $0\leq x\leq\alpha^\prime(p,\beta)\cdot n^{-1/2}$, in the diffusive and critical regime,
    \begin{equation*}\begin{aligned}
        &\quad \quad \frac{\mathbb{P}(M^j_n/\sqrt{w_n}\geq x)}{1-\Phi(x)}=\frac{\mathbb{P}(M^j_n/\sqrt{w_n}\leq -x)}{1-\Phi(-x)}\\
        &=
        \begin{cases}
            C\cdot\exp(\tfrac{x^3}{\sqrt{n}}+\frac{x^2}{n}+\tfrac{1}{\sqrt{n}}(1+\tfrac{1}{2}\log n)(1+x)) & \text{when}\quad a<1-\frac{1}{2(\beta+1)}\\
            C\cdot\exp(\tfrac{x^3}{\sqrt{n}}+\tfrac{x^2}{\log n}+(\tfrac{1}{\sqrt{\log n}}+\tfrac{1}{2\sqrt{n}}\log n)(1+x))& \text{when}\quad a=1-\frac{1}{2(\beta+1)}.
        \end{cases}
    \end{aligned}\end{equation*}
\end{lemma}
\begin{lemma}\label{lem: 7.0.6}
    There exists an absolute constant $\alpha^{\prime\prime}(p,\beta)>0$ depending only on $p,\beta$ such that for all $j=1,\ldots,d$ and all $0\leq x\leq\alpha^{\prime\prime}(p,\beta)\cdot n^{-1/2}$, in the diffusive and critical regime,
    \begin{equation*}\begin{aligned}
        &\quad \quad \frac{\mathbb{P}(a_n\mu_nN^j_n/\sqrt{w_n}\geq x)}{1-\Phi(x)}=\frac{\mathbb{P}(a_n\mu_nN^j_n/\sqrt{w_n}\leq -x)}{1-\Phi(-x)}\\
        &=
        \begin{cases}
            C\cdot\exp(\tfrac{x^3}{\sqrt{n}}+\tfrac{x^2}{n^{2(1-a)(\beta+1)}}+\tfrac{1}{\sqrt{n}}(n^{1/2-(1-a)(\beta+1)}+\tfrac{1}{2}\log n)(1+x)) & \text{when}\quad a<1-\frac{1}{2(\beta+1)}\\
            C\cdot\exp(\tfrac{x^3}{\sqrt{n}}+\tfrac{x^2}{n\log n}+(\tfrac{1}{\sqrt{n\log n}}+\tfrac{1}{2\sqrt{n}}\log n)(1+x)) & \text{when}\quad a=1-\frac{1}{2(\beta+1)}.
        \end{cases}
    \end{aligned}\end{equation*}
\end{lemma}
\bigskip
\subsection*{Acknowledgements} The authors wish to thank Jean Bertoin and Pierre Tarres for numerous discussions and insightful comments.

\bigskip

\bibliographystyle{plain}
\bibliography{literature}
\begin{spacing}{1}

\end{spacing}

\end{document}